\documentclass[11pt]{amsart} \textwidth=14.5cm \oddsidemargin=1cm
\usepackage{amsmath,amsthm,amssymb,latexsym,epic,bbm,comment,yfonts,mathrsfs}
\usepackage{graphicx,enumerate,stmaryrd,rotating,xcolor}
\usepackage[all]{xy}
\xyoption{2cell}

\allowdisplaybreaks

\usepackage[active]{srcltx}
\usepackage{youngtab}

\usepackage[latin1]{inputenc}
\usepackage{amsxtra}
\usepackage{amsmath}
\usepackage{amscd}
\usepackage{amssymb}
\usepackage{amsfonts}
\usepackage[all]{xy}
\usepackage{mathrsfs}
\usepackage{amsthm}
\usepackage{color}
\usepackage{upgreek}
\usepackage{verbatim}  
\usepackage{enumerate}
\usepackage{latexsym}
\usepackage{graphicx}
\usepackage{tikz}
\usepackage{todonotes}
\usepackage{setspace}
\usepackage{hyperref}
\usepackage{stmaryrd}
\usepackage{nicefrac}
\usepackage{euscript}
\usetikzlibrary{decorations.pathmorphing}

 \definecolor{darkgreen}{HTML}{336633}
 \definecolor{darkred}{HTML}{993333}
\makeatletter

\newcommand{\arxiv}[1]{\href{http://arxiv.org/abs/#1}{\tt
    arXiv:\nolinkurl{#1}}}

\theoremstyle{plain}
\newtheorem{thm}{Theorem}
\newtheorem*{thm*}{Theorem}
\newtheorem*{thmA}{Theorem A}
\newtheorem*{thmB}{Theorem B}
\newtheorem*{thmC}{Theorem C}

\newtheorem{lem}[thm]{Lemma}
\newtheorem{prop}[thm]{Proposition}

\newtheorem{cor}[thm]{Corollary}
\newtheorem{df-prop}[thm]{Definition-Proposition}

\theoremstyle{definition}

\theoremstyle{remark}
\newtheorem{rem}[thm]{Remark}
\newtheorem{ex}[thm]{Example}

\def\mod{\operatorname{-mod}\nolimits}

\def\Hom{\operatorname{Hom}\nolimits}
\def\End{\operatorname{End}\nolimits}
\def\Res{\operatorname{Res}\nolimits}
\def\Ind{\operatorname{Ind}\nolimits}
\def\Ext{\operatorname{Ext}\limits}

\def\wt{\operatorname{wt}\nolimits}

\def\pr{\operatorname{pr}\nolimits}

\def\gl{\mathfrak{gl}}

\def\la{\lambda}

\def\pn{\mf{p} (n)}

\def\ov{\overline}

\newcommand{\mc}{\mathcal}
\newcommand{\mf}{\mathfrak}
\newcommand{\C}{\mathbb C}

\newcommand{\oa}{{\bar 0}}
\newcommand{\ob}{{\bar 1}}
\newcommand{\vare}{\epsilon} 

\newcommand{\ad}{\mathrm{ad}}

\newcommand{\fg}{\mathfrak{g}}

\newcommand{\fh}{\mathfrak{h}}

\newcommand{\mC}{\mathbb{C}}

\newcommand{\h}{\mathfrak{h}}
\newcommand{\N}{\mathbb{Z}_{\geq 0}}
\newcommand{\ch}{\mathrm{ch}}

\newcommand{\coker}{{\rm coker}}

\newcommand{\g}{\mathfrak{g}}

\newcommand{\Real}{\mathrm{Re}}

\newcommand{\Z}{{\mathbb Z}}

  \newcommand{\Whoa}{{\text{Wh}_\zeta}}

       \def\ups{{\Upsilon}}

                   \def\mofB{{\text{mof-}B}}

         \def\Lnua{{\Lambda(\zeta)}}

             \def\fdWmod{  W_\zeta\text{-fdmod}}
             \def\Mod{\operatorname{-Mod}\nolimits}

                \def\Whz{\text{Wh}_\zeta}
                 \def\str{\text{str}}

\begin{document}

\numberwithin{equation}{section}

\title[Whittaker categories and finite W-superalgebras]{Whittaker categories of quasi-reductive Lie superalgebras and principal finite W-superalgebras}

\author[Chen]{Chih-Whi Chen}
\address{Department of Mathematics, National Central University, Chung-Li, Taiwan 32054} \email{cwchen@math.ncu.edu.tw}
\author[Cheng]{Shun-Jen Cheng}
\address{Institute of Mathematics, Academia Sinica, and National Center of Theoretical Sciences, Taipei, Taiwan 10617} \email{chengsj@math.sinica.edu.tw}
\date{}

\begin{abstract} We study the Whittaker category $\mc N(\zeta)$ of the Lie superalgebra $\g$ for an arbitrary character $\zeta$ of the even subalgebra of the nilpotent radical associated with a triangular decomposition of $\g$.
We prove that the Backelin functor from either the integral subcategory or any strongly typical block of the BGG category to the Whittaker category sends irreducible modules to irreducible modules or zero.
The category $\mc N(\zeta)$ provides a suitable framework for studying finite $W$-superalgebras associated with an even principal nilpotent element. For the periplectic Lie superalgebras $\pn$, we formulate the principal finite $W$-superalgebras $  W_\zeta$ and establish a Skryabin-type equivalence. For a basic classical and a {\em strange} Lie  superalgebras, we prove that the category of finite-dimensional modules over a given principal finite $W$-superalgebra $W_\zeta$ is equivalent to  $\mc N(\zeta)$ under the Skryabin equivalence, for a non-singular character $\zeta$. As a consequence, we give a super analogue of Soergel's Struktursatz for a certain Whittaker functor from the integral BGG category $\mc O$ to the category of finite-dimensional modules over $W_\zeta$.
\end{abstract}

\maketitle

\setcounter{tocdepth}{1}
\tableofcontents


\vspace{5mm}

 \section{Introduction} \label{sect::intro}

 \subsection{Setup}

 	\subsubsection{Lie superalgebras}
  In the present paper, we are mainly interested in the following Lie superalgebras $\g$  over $\C$ from Kac's list \cite{Kac77}:
  \vskip0.2cm
  \leftline{$\text{(Basic classical)}~\mathfrak{gl}(m|n),\,\,\mathfrak{sl}(m|n),\,\,\mathfrak{psl}(n|n),\,\,\mathfrak{osp}(m|2n),\,\, D(2,1;\alpha),\,\, G(3),\,\, F(4),$}
  \vskip0.2cm
 \leftline{$\text{(Strange)}\hskip0.4cm\mathfrak{p}(n),\,\,[\pn,\pn],\,\,\mathfrak{q}(n),\,\,\mathfrak{sq}(n),\,\, \mathfrak{pq}(n)\mbox{ and }\mathfrak{psq}(n).$}
   \vskip0.2cm


In particular, $\g$ is quasi-reductive, i.e., $\g_\oa$ is a reductive Lie algebra and $\g_\ob$ is semisimple as an $\ad\g_\oa$-module.  Fix a Cartan subalgebra $\h_{\oa}$ of $\mf g_\oa$. We then have a root space decomposition
$\fg=\bigoplus_{\alpha\in\Phi\amalg \{0\}}\fg^\alpha,~\mbox{with }\;\fg^\alpha=\{X\in\fg\,|\, [h,X]=\alpha(h)X,\;\forall h\in\fh_{\oa}\},$
where  $\Phi\subseteq\fh_{\oa}^\ast$ is the set of roots.
In this paper, we fix a triangular decomposition
\begin{align}
	&\g=\mf n^- \oplus \h \oplus \mf n \label{eq::tri}
\end{align}  in the sense of \cite[Section 2.4]{Ma} (see also \cite[Section 1.4]{CCC}), where $\fh:=\g^0 =\bigoplus_{\Real \alpha(H)=0} \fg^\alpha,\quad \mf n:=\bigoplus_{\Real \alpha(H)>0} \fg^\alpha,$ and $\mf n^-:=\bigoplus_{\Real \alpha(H)<0} \g^\alpha,$ for some $H\in\fh_{\oa}$.  We define the {\em Borel subalgebra} $\mf b = \mf h+\mf n$; see also \cite[Section 3.2, Section 3.3]{Mu12}.  The subalgebra  $\fh$ is referred to as the {\em  Cartan subalgebra} of $\fg$.   Note that $\mf h=\mf h_\oa$ except when $\mf g$ is of queer type. We note that every Borel subalgebra of $\g_\oa$ is an underlying even subalgebra of a Borel subalgebra of $\g$.

\subsubsection{Simple Whittaker modules}\label{sect::112}
A finitely generated $\g$-module $M$ is called a {\em Whittaker module} if it is locally finite over $Z(\g_\oa)$ and $U(\mf n)$. A character $\zeta$ of $\mf n_\oa$ is called {\em non-singular} if  $\zeta(\g_\oa^{\alpha})\neq 0$ for any (even) simple root $\alpha$. In the case of Lie algebras, the study of Whittaker modules was initiated by Kostant in \cite{Ko78}, where Whittaker modules associated to  non-singular characters $\zeta$ were considered. Since then, there has been considerable progress on Whittaker modules for reductive Lie algebras; see, e.g., \cite{Ly, Mc, Mc2, MS, B, BM, CM21, Bro, R1, AB}.

 Denote by $\ch \mf n_\oa:= (\mf n_\oa/[\mf n_\oa,\mf n_\oa])^\ast$  the set of characters on $\mf n_\oa$. Fix a character $\zeta\in \ch \mf n_\oa$.   We define a set     ${\Pi_\zeta}:=\{\alpha \in \Phi_\oa^+|\zeta(\mf g_\oa^\alpha)\neq 0\}$, where ${\Phi_\oa^+}$ denotes the set of positive even roots. Define the Levi subalgebra $\mf l_\zeta$ in a parabolic subalgebra $\mf p$ of $\g_\oa$ generated  by $\mf h$ and  $\mf g_\oa^{\pm\alpha}$, for $\alpha\in {\Pi_\zeta}$.  The Weyl group  of $\mf l_\zeta$ is denoted by  $W(\mf l_\zeta)$. Following Mili{\v{c}}i{\'c} and Soergel \cite{MS},  we consider the category $\mc N$ of Whittaker modules. Then $\mc N$ decomposes into a direct sum of full subcategories $\mc N(\zeta)$ of Whittaker modules on which $x-\zeta(x)$ acts locally nilpotently, for any $x\in \mf n_\oa$.

In the case when $\g=\g_\oa$ is a reductive Lie algebra, we recall  the {\em standard Whittaker modules} $M(\la,\zeta)$ ($\la\in \h^\ast$) in $\mc N(\zeta)$:
\begin{align}
	&M(\la, \zeta):= U(\g)\otimes_{U(\mf p)} Y_\zeta(\la, \zeta).\label{def::standardWhi}
\end{align}
Here $Y_\zeta(\la, \zeta):=U(\mf l_\zeta)/(\text{Ker}\chi_\la^{\mf l_\zeta}) U(\mf l_\zeta)\otimes_{U({\mf n}\cap \mf l_\zeta)}\mathbb C_\zeta$ denotes Kostant's simple Whittaker modules from \cite{Ko75}, where $\chi_\la^{\mf l_\zeta}$ is the central character of $\mf l_\zeta$ associated to $\la$ and $\C_\zeta$ is the one-dimensional ${\mf n}\cap \mf l_\zeta$-module induced by the character $\zeta|_{{\mf n}\cap \mf l_\zeta}$. Up to isomorphisms, they are parametrized by the characters $\zeta$ and the coset representatives of $W(\mf l_\zeta)$ under the dot-action of Weyl group. The study of standard Whittaker modules goes back to the works of McDowell \cite{Mc,Mc2} and Mili{\v{c}}i{\'c}-Soergel \cite{MS}.  They plays an analogous role of Verma modules in the representation theory of $\mc N(\zeta)$, for example,  simple objects in $\mc N(\zeta)$ are classified by the tops of standard Whittaker modules $M(\la,\zeta)$.


 Recently, various aspects of the category $\mc N(\zeta)$ of Whittaker modules over Lie superalgebras have been investigated (see, e.g.,  \cite{BCW,Ch21,Ch212,CCM2,CC22}). In particular,
 the construction of standard Whittaker modules affords a superalgebra generalization in the case when $\mf l_\zeta$ is a Levi subalgebra of $\g$; see \cite[Section 3.1]{Ch21} (see also \cite{CC22}), including type I Lie superalgebras (see also \eqref{eq::typeI} for a list). In this case,  many aspects of the approach of using standard Whittaker modules have been extended to Lie superalgebras. 
   However, the case when $\mf l_\zeta$ is not a Levi subalgebra of $\g$ does not fit into such a   framework.

  As in the Lie algebra case (see \cite[Theorem 2.6]{MS}), each object in $\mc N(\zeta)$ has finite length.  For Lie algebras, Backelin in  \cite{B} provided a complete solution to the problem of composition multiplicities for standard Whittaker modules $M(\la,\zeta)$ in terms of
  Kazhdan-Lusztig polynomials; see also \cite{MS} for the case of integral weights $\la$. In particular, he introduced a certain exact functor $\Gamma_\zeta$ from the Bernstein-Gelfand-Gelfand category $\mc O$ to $\mc N(\zeta)$ for reductive Lie algebras,  transforming Verma modules to standard Whittaker modules or zero, and simple modules to simple modules or zero. For Lie superalgebras, in  \cite{Ch21,CC22}  the authors used a super analogue of Backelin's functor to reduce the multiplicity problem of standard Whittaker modules to that of Verma modules in the category $\mc O$. As a consequence, the Backelin functor is further realized as certain Serre quotient functor and this leads to several categorification pictures of the $q$-symmetrizer on the $q$-symmetrized Fock space and the $q$-symmetrizing map for the quantum and $\iota$-quantum groups in \cite{CCM2,CC22}.

 \subsubsection{Finite-dimensional modules for finite $W$-superalgebras}
  Finite W-algebras are certain associative algebras constructed from pairs $(\g,e)$, where  $\g$ is a complex semisimple Lie algebra and $e\in \g$ is a nilpotent element.  The representation theory of finite $W$-algebras  has been studied intensively since Premet's work \cite{Pr02}, see, e.g., \cite{BGK, GG, Lo09, Lo10, Lo11, Pr02,Pr07,Pr07Mo, Pr10}. In particular, their finite-dimensional modules have received a considerable amount of attention. This is mainly due to the connection to primitive ideals for $U(\g)$, see, e.g., \cite{Gi, Lo10, Lo11, Pr07,Pr07Mo, Pr10}.

 The construction of finite $W$-algebras has a natural superalgebra generalization. Some attempts have been made in understanding the representation theory of finite $W$-superalgebras for basic classical and queer Lie superalgebras associated with the regular even nilpotent orbits, which we refer to as the {\em principal finite $W$-superalgebras}.

 In \cite[Theorems 7.2, 7.3]{BBG}, Brown, Brundan and Goodwin  provided a description of all simple modules over principal finite $W$-superalgebras for $\gl(m|n)$. In particular, these simple modules are all finite dimensional. The principal finite $W$-superalgebras associated with the queer Lie superalgebra $\mf q(n)$ has been considered by Serganova and Poletaeva in \cite{PSqn,PS21}. In particular, the simple modules have been classified in \cite[Theorem 4.7, Proposition 4.13]{PS21}. As a consequence, all simple modules are shown to be finite dimensional.
 This leads to a classification of irreducible finite-dimensional modules of the shifted Yangian $Y(\mf{gl}(1|1))$ associated with $\mf{gl}(1|1)$ (cf. \cite[Theorem 4.5]{BBG}) and that of super Yangian $YQ(1)$ associated with $\mf q(1)$ (cf. \cite[Theorem 5.13]{PSqn}). Connections between finite $W$-superalgebras and super Yangians have been investigated by Peng in \cite{Pe, Pe2}.

  In \cite{Ko78}, Kostant   established an equivalence between the category $\mc N(\zeta)$ with non-singular $\zeta$ and the category of finite-dimensional modules over the corresponding principal finite $W$-algebra (see also \cite{MS2}). In this case, the latter is isomorphic to the center $Z(\g_\oa)$ of $U(\g_\oa)$.  Skryabin in \cite{Skr} generalized Kostant's result to arbitrary finite $W$-algebras; see also \cite{GG} and Lemma \ref{Skr}.  A ramification of Skryabin's equivalence, concerning the category $\mc O$ of finite $W$-algebras, is established by Losev in \cite{Lo09}.

An analogue of Skryabin's equivalence also holds for basic classical and queer Lie superalgebras; see, e.g., \cite[Remarks 3.9-3.10]{Zh14},  \cite[Theorem 2.17]{ZS15} and \cite[Theorem 4.1]{SW20}.  However, a precise connection between the category $\mc N(\zeta)$ and the category of finite-dimensional modules over principal finite $W$-superalgebra for Lie superalgebras does not seem to have been described in the literature.

 \subsubsection{Goals}
The goal of this paper is to study several aspects of the Whittaker categories $\mc N(\zeta)$. Namely, the present paper attempts to give a classification of simple Whittaker modules in either the integral blocks or any strongly typical blocks of  $\mc N(\zeta)$ for arbitrary character $\zeta$ on $\mf n_\oa$ and to establish a Skryabin-type equivalence between the category of finite-dimensional modules over the principal finite $W$-superalgebra and the category $\mc N(\zeta)$ for an arbitrary non-singular $\zeta\in\ch \mf n_\oa$.


 \subsection{Main results} To explain the contents of the paper in more detail, we start by   recalling the super analogue of Backelin's functor $\Gamma_\zeta(-):\mc O\rightarrow \mc N(\zeta)$ from \cite[Section 5.2]{Ch21}, which is a naturally extension of Backelin's original functor from \cite{B} for reductive Lie algebras.  For a given module $M\in \mc O$, let $\ov M$ denote the completion of $M$ with respect to its weight space by $\ov M:=\Pi_{\la\in \h^\ast_\oa} M_\la$. Then $\Gamma_\zeta(M)$ is defined as the $\g$-submodule of $\ov M$ consisting of all vectors in $\ov M$ on which $x-\zeta(x)$ acts nilpotently, for all $x\in \mf n_\oa$.

 Assume that $\g$ is a basic classical Lie superalgebra. Our first main result is a classification of simple Whittaker modules in two types of subcategories in $\mc N(\zeta)$, which  we shall explain as follows. First, we consider the category $\mc N(\zeta)_{\chi_\la}$ of all Whittaker modules $M\in \mc N(\zeta)$ annihilated by some power of the kernel of a strongly typical central character $\chi_\la$ in the sense of \cite{Gor2}. Next, for a given weight $\la\in \h^\ast$, we consider the Serre subcategory $\mc N(\zeta)_{\la+\Lambda}$ of $\mc N$ (resp. the Serre subcategory $\mc O_{\la +\Lambda}$ of $\mc O$) generated by composition factors of $\Gamma_\zeta(L(\mu))$ (resp. $L(\mu)$) for $\mu\in \la+\Lambda$ (see also Proposition \ref{prop::3}), where $\Lambda\subseteq \mf h^\ast$ denotes the set of integral weights.

 \subsubsection{Classification of simple Whittaker modules}

 For any $\mu\in \mf h^\ast$,
 we denote by $M(\mu)$ the Verma module with highest weight $\mu$ and by $L(\mu)$ its unique
 simple quotient. The Weyl group $W$ of $\fg$, which is the Weyl group of $\g_\oa$, acts naturally on $\mf h^\ast$  via the two dot actions $\cdot$ and $\circ$; see \eqref{dotaction}--\eqref{dotaction2}.  Our first main result  is the following.

 \begin{thmA}[Proposition \ref{lem:gorelikeq},~ Theorem \ref{thm::4}] Let $\g$ be any basic classical Lie superalgebra.
  We have
 	\begin{itemize}
 		\item[(1)] Suppose that $\la{\in \h^\ast}$ is strongly typical. Then the isomorphism classes in the set $\{\Gamma_\zeta(L(w\circ \la))|~w\in W\}$  is a complete list of simple modules in $\mc N(\zeta)_{\chi_\la}$.
 		\item[(2)]
 			Suppose that $\nu \in \mf h^\ast$ is such that $\nu$ is dominant with respect to the  dot-action  $\cdot$ of $W$ and its stabilizer subgroup coincides with $W(\mf l_\zeta)$.
 			Then the Backelin functor $\Gamma_\zeta(-): \mc O_{{\nu}+\Lambda}\rightarrow \mc N(\zeta)_{{\nu} +\Lambda}$ transforms simple modules to simple modules or zero. 		Furthermore, the set
 		\begin{align}
 			&\{\Gamma_\zeta(L(\mu))|~\mu\in \Lambda({\nu})\},	
 		\end{align} is an exhaustive list of mutually non-isomorphic simple Whittaker modules in $\mc N(\zeta)_{{\nu} +\Lambda}$.   Here $\Lambda({\nu})$ denotes the set of weights  $\mu\in{\nu}+\Lambda$ such that every non-zero root vector in $\g_\oa^{-\alpha}$ acts on $L(\mu)$ freely, for any simple root $\alpha$ in $\mf l_\zeta$.
 	\end{itemize}
 \end{thmA}

 We remark that an explicit description of the weights in $\Lambda({\nu})$ can be found in \cite[Section 4.3]{CCC}. Theorem A Part (2) extends the case when $\g$ is of type I, including the case of $\pn$, in \cite[Theorem 20]{Ch21} and the case when $\mf l_\zeta$ is a Levi subalgebra of $\g$ in \cite[Theorem 1]{CC22}. However, the results in these cases in loc.~cit.~apply to more general weights.

In \cite{FM09}, Frisk and Mazorchuk investigated the {\em regular strongly typical blocks} in the category $\mc O$ for the queer Lie superalgebra $\mf q(n)$ and established an equivalence of such blocks to the corresponding blocks of the category $\mc O$ for $\mf{gl}(n)$. This equivalence allows us to prove an analogous version of Theorem A Part (2) for $\mf q(n)$; see Appendix \ref{sect::appA}.

\subsubsection{A Skryabin-type equivalence}

As said above, finite $W$-superalgebras associated with basic classical and queer Lie superalgebras have been studied. However, to the best of our knowledge a version of finite W-superalgebra associated with the periplectic Lie superalgebras $\pn$ has not been studied in the literature. In the present paper, we formulate a principal finite $W$-superalgebra  $W_\zeta$ associated with a nilpotent element $e$ in $\g_\oa$ in the spirit of Premet. We then establish a Skryabin-type equivalence for the periplectic Lie superalgebras.

To explain our next main result in more detail, we first recall the construction of principal finite $W$-superalgebras. Let $\g$ be either a basic classical Lie superalgebra or a {\em strange} Lie superalgebra, i.e., $\pn$ or $\mf q(n)$, with a given principal nilpotent element $e\in \g_\oa$ inside an $\mf{sl}(2)$-triple $\langle e,f,h\rangle$. This leads to a {nilpotent subalgebra $\mf m$ inside a Borel subalgebra $\mf b$ of $\g$} and a character $\zeta:\mf m\rightarrow \C$ {which defines a one-dimensional $\mf m$-module $\C_\zeta$,} see Sections \ref{sect::DefWal}, \ref{sect::skr} and \ref{sect::defWalgpn}.  The principal finite $W$-superalgebra $W_\zeta$ is defined as the opposite $\mf g$-endomorphism algebra of the generalized Gelfand-Graev module $Q_\zeta := U(\g)\otimes_{U(\mf m)}\C_\zeta$. Let $\Whz(-)$ denote the functor of taking Whittaker vectors, i.e., $\Whz(M)$ denotes the subspace of vectors $v\in M$ on which $x-\zeta(x)$ acts trivially, for any $x\in \mf m$.  By Skryabin's  equivalence, this leads to mutually inverse equivalences  $Q_\zeta\otimes_{W_\zeta}-$ and $\Whz(-)$ between the category of $W_\zeta$-modules and the category of $\g$-modules on which $x-\zeta(x)$ acts locally nilpotently, for any $x\in \mf m$. Let ${W}_\zeta\text{-fdmod}$ denote the category of finite-dimensional $W_\zeta$-modules. Our second main result is the following.
\begin{thmB}[Theorems \ref{thm::eqvthm}, \ref{thm::eqvthmpn} and \ref{thm::pnmain}] Suppose that  $\g$ is   either  a basic classical Lie superalgebra or a strange Lie superalgebra   with $\zeta$ defined as above.
		Then the Whittaker functor $\emph{Wh}_\zeta(-)$ restricts to an equivalence from $\mc N(\zeta)$ to ${W}_\zeta\emph{-fdmod}$ with quasi-inverse $Q_\zeta\otimes_{ W_\zeta}(-)$. In particular, all simple $W_\zeta$-modules are finite dimensional.
\end{thmB}

 As an illustration consider $\g=\pn$ and let $M(\la,\zeta)$ denote the standard Whittaker modules of $\mf p(n)$. Then it follows from Theorem B the following set   $$\{\Whz(M(\la,\zeta))|~\la\in \h^\ast\text{ is anti-dominant}\},$$ is a complete list of mutually non-isomorphic simple modules over $W_\zeta$. In particular, they are all finite dimensional.

   We remark that the Whittaker categories $\mc N(\eta)$ with respect to different Borel subalgebras $\mf b$ and non-singular characters $\eta$ are all equivalent; see Appendix \ref{sect::app}. This allows to realize them as the category ${W}_\zeta\text{-fdmod}$ of finite-dimensional modules over $W_\zeta$.

 	\subsubsection{Whittaker functors from $\mc O_{\Lambda}$ to $W_\zeta\emph{-fdmod}$}
   Brundan and Goodwin \cite{BG2} constructed a certain {\em Whittaker coinvariants functor} $H_0$ for the general linear Lie superalgebra $\mf{gl}(m|n)$, which is an exact functor from $\mc O_\Lambda$ to $W_\zeta\text{-fdmod}$. This functor satisfies properties similar to Soergel's combinatorial functor $\mathbb V$ for the category $\mc O$ of a semisimple Lie algebra. In addition, the functor $H_0:\mc O_{\Lambda}\rightarrow H_0(\mc O_\Lambda)$ is a realization of a certain Serre quotient functor  in the sense of \cite{Gabriel}; see \cite[Theorem 4.8]{BG2}. A more general version has been investigated by Mazorchuk and the authors in \cite[Remark 46, Corollary 47]{CCM2} for Lie superalgebras of type I. Our third main result is to provide a further extension of this result that includes all the basic classical and the strange Lie superalgebras $\pn$ and $\mf q(n)$.
 	
 	   To explain this in more detail, we let $\nu\in {\mf h^\ast_\oa}$ be an integral weight on which the dot-action of the Weyl group is trivial. {Let $M_0(\nu,\zeta)$ denote the corresponding standard Whittaker module over $\g_\oa$ from \eqref{def::standardWhi}.} We introduce the full subcategory  	   $\fdWmod^1_\Lambda$ of $\fdWmod$, which is defined as the category of $W_\zeta$-modules  isomorphic to  subquotients of $\Whz(\Ind (E\otimes M_0(\nu,\zeta)))$ for any finite-dimensional $\g_\oa$-module $E$. Alternatively, $\fdWmod^1_\Lambda$ is equal to $\Whz(\Gamma_\zeta(\mc O_\Lambda))$. Let $\ov{\mc O}_\Lambda$ denote the Serre quotient category of $\mc O_\Lambda$ (in the sense of \cite[Chapter III]{Gabriel}) by the Serre subcategory generated by all simple modules $L(\la)$ such that their projective covers $P(\la)$ in $\mc O_\Lambda$ are not injective. Let $\pi: \mc O_\Lambda\rightarrow \ov{\mc O}_\Lambda$ be the corresponding Serre quotient functor (see also \cite[Section 4.2]{CCM2}).  The following theorem establishes an analogue of Soergel's Struktursatz for the composition of functors $\text{Wh}_\zeta\circ \Gamma_\zeta(-)$ from $\mc O_\Lambda$ to $W_\zeta\text{-fdmod}_{\Lambda}$.
 	
\begin{thmC}
	
	The functor $\emph{Wh}_\zeta\circ \Gamma_\zeta(-): \mc O_\Lambda\rightarrow \text{\em $\fdWmod^1_\Lambda$}$ satisfies the universal property of the Serre quotient category of $\mc O_{\Lambda}$ by the Serre subcategory generated by simple modules $L(\la)$ projective covers of which are not injective. This induces an equivalence $E$ making the following diagram commute:
	$$\xymatrixcolsep{3pc} \xymatrix{ 		&	\mc O_\Lambda  \ar[ld]_-{\pi(-)}  \ar@<-2pt>[rd]^{ \emph{Wh}_\zeta\circ \Gamma_\zeta(-)} &       \\	\ov{\mc O}_\Lambda	 \ar[rr]_-{E(-)}^{\cong} & &  \text{\em $\fdWmod^1_\Lambda$}}$$ Furthermore,  $\emph{Wh}_\zeta\circ \Gamma_\zeta(-)$ is full and faithful on projective  modules in $\mc O_\Lambda$.
\end{thmC}

Theorem C extends \cite[Remark 46, Corollary 47]{CCM2}, where the case of $\g=\gl(m|n)$ of Brundan and Goodwin \cite{BG2} was discussed, to all basic classical and strange Lie superalgebras. The classification of projective-injective modules in $\mc O_\Lambda$ has been given in \cite[Section 4.3]{CCC}.

  It is  worth pointing out the connection between the functor $\text{Wh}_\zeta\circ \Gamma_\zeta(-)$ and Soergel's combinatorial functor $\mathbb V$ from \cite{So90}  in the setting of category $\mc O$ for a reductive Lie algebra $\g=\g_\oa$. The latter, playing an significant role in the representation theory of Lie algebras, is an exact functor $\mathbb V(-):=\Hom_{\g}(P(\la), -)$ from the block in $\mc O_\Lambda$ containing $L(\la)$ to $W_\zeta$-$\text{fdmod}$. Backelin in \cite[Corollary 5.4]{B} initially proved that $\Whz\circ \Gamma_\zeta(-)$ coincides $\mathbb V(-)$; see also \cite[Proposition 2]{AB}. Arakawa in \cite[Theorem 2.6.1]{Ara} described a natural isomorphism between these functors and gave another proof of this result. Independently, Mazorchuk and the authors developed in \cite[Section 7.4.3, Section 9.2]{CCM2} the super analogue $\mathbb V^{sup}$ of Soergel's functor $\mathbb V$ and proved that it satisfies the same universal property as described in Theorem C.  The arguments therein can be generalized to any quasi-reductive Lie superalgebras, including reductive Lie algebras. As a consequence, the functor $\text{Wh}_\zeta\circ \Gamma_\zeta(-)$ can be identified with the super Soergel's functor $\mathbb V^{sup}$, up to an equivalence between their target categories, for any Lie superalgebras considered in the present paper.

 \subsection{Structure of the paper}
 This paper is organized as follows.  In Section \ref{sect::Pre}, we provide some background materials on Lie superalgebras. In particular, we review the BGG category $\mc O$, Gorelik's equivalence of categories for strongly typical blocks and the notion of category with full projective functors. Section \ref{sect::WhCat} is devoted to the proofs of Theorem A.  We establish
 in Section \ref{sect::stblock} Part (1) of Theorem A. Section \ref{sect::33} offers a description of block decomposition of $\mc N(\zeta)$.
 The  necessary preliminaries for the connection between categories $\mc O_{{\nu}+\Lambda}$ and $\mc N(\zeta)$ are gathered in Section \ref{Sect::34}, which is then used to prove Part (2) of Theorem A.

  In Section \ref{sect::FWalg}, finite $W$-superalgebras for basic classical Lie superalgebras are introduced. Theorem B for this case is established in Theorem \ref{thm::eqvthm}.  We give a definition of the principal finite $W$-superalgebra of $\pn$ in Section  \ref{sect::pn}. The proofs of Theorem B for $\mf q(n)$ and $\pn$ are given in Sections \ref{sect::511} and \ref{sect::54}, respectively.  A description of the block decomposition for $W_\zeta$-fdmod is given in Section \ref{sect::61}. In
   Section \ref{sect::6} we discuss some consequences of Theorem B and give a proof of Theorem C.
In Appendix A, we establish analogue of Theorem A Part (2) for $\mf q(n)$, which we then use to complete the proof of Theorem C for this case.
  Finally, we prove in Appendix \ref{sect::app} equivalence of categories between Whittaker categories $\mc N(\zeta)$ with respect to different Borel subalgebras $\mf b$ and characters $\zeta\in \ch\mf n_\oa$. As a consequence, this allows us to realize the category $\mc N(\zeta)$ as the category of finite-dimensional modules over principal finite $W$-superalgebra of $\g$ for an arbitrary non-singular character $\zeta\in\ch\mf n_\oa$.

\vskip 0.5cm

{\bf Acknowledgment}. The authors are partially supported by National Science and Technology Council grants of the R.O.C., and they further acknowledge support from the National Center for Theoretical Sciences.  We thank  Yang Zeng for useful discussions.

 \section{Preliminaries} \label{sect::Pre}


   \subsection{Weyl group and weights} \label{sect::11}

We denote the sets of positive and negative roots corresponding to $\mf b$ by $\Phi^+$ and $\Phi^-$, respectively.
The sets of even and odd roots are denoted respectively by $\Phi_{\oa}$ and $\Phi_{\ob}$ with similar notations ${\Phi_\oa^\pm, \Phi_\ob^\pm}$ for even and odd positive and negative roots. The subsets of simple positive roots in $\Phi$ and $\Phi_\oa$ are denoted  by $\Pi$ and $\Pi_\oa$, respectively.  We have $\Pi_{\oa}\not=\Pi\cap \Phi_{\oa}$ in general.

The Weyl group $W$ of $\g$ is defined as the Weyl group $W$ of $\g_\oa$ with its defining
action on $\h_\oa^\ast$. We define two  dot-actions of $W$ on $\h_\oa^\ast$ as follows \begin{align}
	&w\cdot \la = w(\la+\rho_\oa) - \rho_\oa, \label{dotaction}\\
	&w\circ \la = w(\la+\rho)-\rho, \label{dotaction2}
\end{align} for any $\la \in {\h_\oa^\ast}$.  Here $\rho_\oa := \frac{1}{2}\sum_{\alpha\in \Phi_\oa^+}\alpha,~\rho := \frac{1}{2}(\sum_{\alpha\in \Phi_\oa^+}\alpha  -\sum_{\beta\in \Phi_\ob^+}\beta)$.

Throughout, we fix a non-degenerate $W$-invariant bilinear form $\langle \cdot, \cdot \rangle$ on $\h^\ast_\oa$,  which we assume to be induced from an  even  non-degenerate invariant supersymmetric bilinear form on $\g$  if the latter exists. For a given $\alpha\in \Phi^+_{\oa}$, we let $\alpha^\vee=2\alpha/\langle\alpha,\alpha\rangle$.

A weight in $\h_\oa^\ast$ is called {\em integral}, {\em regular},  {\em dominant} or {\em anti-dominant} if it is integral, regular, dominant
or anti-dominant as a $\g_\oa$-weight, respectively.  Denote by $\Lambda$ the set of all integral weights. 

 \subsection{Representation categories} \label{sect::23}
 \subsubsection{Central blocks, induction and restriction functors}
   We denote by $U(\g)$  the universal enveloping algebra of $\g$, and by $Z(\g)$ the center of $U(\g)$. Let $\g\Mod$ and $\g\mod$ denote the category of all $\g$-modules and finitely-generated $\g$-modules,  respectively. Denote by $\g\mod_{Z(\g)}$  the full subcategory of $\g\mod$ on which $Z(\g)$ acts locally finitely. For central characters $\chi: Z(\g)\to\mC$ and $\chi^0: Z(\g_\oa)\to\mC$,  we denote by $\g\mod_\chi$ and $\g_\oa\mod_{\chi^0}$ the full subcategories of objects in $\g\mod_{Z(\g)}$ and $\g_\oa\mod_{Z(\g_\oa)}$ annihilated by some powers of $\ker(\chi)$ and $\ker(\chi^0)$, respectively. The endofunctor $(-)_\chi$ of $\g\mod_{Z(\g)}$ is defined as taking the largest direct summand in the block $\g\mod_\chi$. For a given weight $\la\in\h_\oa^\ast$, we denote by $\chi_\la$  (resp. $\chi^0_{\la}$) the central character of  $\g$ (resp. $\g_\oa$) associated to $\la$.

  For a subalgebra $\mf s\subseteq\g$, we denote by $\Res^{\mf g}_{\mf s}$ the restriction functor from $\g\mod$ to $\mf s\mod$. It has a  left adjoint functor 	$\Ind^{\mf g}_{\mf s}=U(\mf g)\otimes_{U(\mf s )}(-):~\mf s\mod\rightarrow\g\mod.$
   When $\mf g_\oa\subseteq\mf s$, then $\Ind^{\mf g}_{\mf s}$ is also right adjoint to the $\Res^{\mf g}_{\mf s}$, up to tensoring with a one-dimensional $\mf s$-module; see, e.g., \cite[Theorem~2.2]{BF} and  \cite[Section~2.3.5]{Gor3}. Finally, we define $\Ind: =\Ind_{\g_\oa}^{\g}$ and $\Res : = \Res_{\g_\oa}^\g$.

   A finite-dimensional Lie superalgebra $\g$ is said to be of {\em type I} if $\g$ has  a $\Z_2$-compatible $\mathbb Z$-gradation of the form $\g= \g_{-1}\oplus \mf g_0 \oplus \g_1$ with $\g_{\bar 0}=\g_0$, $\g_{\bar 1}=\mf g_{-1}\oplus \g_1$.  This includes reductive Lie algebras, and \begin{align}
   	&\mathfrak{gl}(m|n),\,\,\mathfrak{sl}(m|n),\,\,\mathfrak{psl}(n|n),\,\,\mathfrak{osp}(m|2n),\,\,\mf p(n). \label{eq::typeI}
   \end{align} In this case, for any $N\in\g_\oa\mod$ we may extend $N$ trivially to a $\g_0+\g_1$-module and define the {\em Kac $\g$-module} $K(M):=\Ind_{\mf g_0+\g_1}^{\g}(M)$. This defines the exact {\em Kac functor} $K(-):\g_\oa\mod\rightarrow\g\mod$.

    \subsubsection{Gorelik's equivalence} \label{sect::Goreq}
     Assume that $\g$ is basic classical.  A weight $\la\in \h^\ast$ is called {\em typical}   if it satisfies $\langle \lambda+\rho,\beta\rangle\not=0$ for any isotropic root $\beta\in\Phi^+_{\ob}$, and is called {\em atypical} otherwise. A weight $\la\in \h^\ast$ is called {\em strongly typical}   if it satisfies $\langle \lambda+\rho,\beta\rangle\not=0$ for any odd root $\beta\in\Phi^+_{\ob}$ (see \cite{Gor2}). If  $\mf g\neq \mathfrak{osp}(2n+1|2m), G(3)$, then  typical weights and strongly typical weights are the same.

    	  For a given central character $\chi: Z(\g)\to\mC$,
    	  the corresponding block $\g\mod_\chi$ is said to be {\em strongly typical} if   $\chi = \chi_{\la}$ for some  strongly typical weight $\la\in \h^\ast$.  Gorelik proved in \cite{Gor2} that   every strongly typical block  $\g\mod_{\chi}$   is equivalent to a block in $\g_\oa\mod$. More precisely, for a given strongly typical central character $\chi: Z(\g)\to\mC$ there exists a {\em perfect mate} in the sense of \cite{Gor2}, which is a central character $\chi^0: Z(\g_\oa)\to\mC$ such that $\g\mod_\chi$ and $\g_\oa\mod_{\chi^0}$ are equivalent.
   \begin{lem}[Gorelik]\label{lem:gorelik}  Let $\chi$ be a strongly typical central character with a perfect mate $\chi^0$.
   	The following functors
   	\begin{align}
   		&\Res(-)_{\chi^0},~\Ind(-)_{\chi},
   	\end{align} give rise to a mutually inverse equivalences of central blocks $\g\mod_{\chi}$ and $\g_\oa\mod_{\chi^0}$.
   \end{lem}


Assume that $\g$ is either basic classical or queer. Following  \cite[Section 5.3]{MaMe12}, a weight $\la\in \h_\oa^\ast$ is said to be {\em generic} if $\la$ is strongly typical, regular and dominant with respect to the dot-action of $W$ and $\Res M(\la)$ is a direct sum of Verma modules with non-isomorphic direct summands of $\Res M(\la)$ corresponding to different central characters.

 \subsubsection{BGG category $\mc O$ }
The   category $\mc O$ associated to the triangular decomposition in \eqref{eq::tri}  is defined as the full subcategory of $\g\mod$ consisting of all finitely-generated $\g$-modules on which $\h_\oa$ acts semisimply and $\mf b$ acts locally finitely.  

Let $M(\la)$ be the Verma module of highest weight $\la$ with respect to the triangular decomposition \eqref{eq::tri} and $L(\la)$ its unique simple quotient.
 The indecomposable projective cover of $L(\lambda)$ in $\mc O$ is denoted by $P(\lambda)$.  We denote by $\mc O^{\oa}$ the BGG category for $\g_{\oa}$ in the sense of \cite{BGG76}. Similarly, we denote the Verma module of highest weight $\la$, its simple quotient and projective cover by $M_0(\la)$, $L_0(\lambda)$ and $P_0(\lambda)$, respectively.


 We define by $\mc F$ the full subcategory of $\mc O$ of finite-dimensional modules. By a projective functor on $\g\mod_{Z(\g)}$ we mean a direct summand of a functor of the form $V\otimes -$, where $V\in \mc F$.  We denote the category of projective functors is denoted by ${\mc Proj}$.  For the case of semisimple Lie algebra $\g=\g_\oa$, the projective functors have been studied by Bernstein and Gelfand  \cite{BG}.

  Put $\la \in \h^\ast_\oa$ to be a generic weight; see Section \ref{sect::Goreq}.    Recall that  $\mc O_{\la +\Lambda}$ denotes the block of $\mc O$ of modules having weights lying in $\la+\Lambda.$ By \cite[Theorem 5.1]{MaMe12} the triple ($\mc O_{\la+\Lambda}$,\,$M(\la)$,\,${\mc Proj}$) forms a {\em category with full projective functors}  in the sense of \cite[Definition 1]{Kh}; see also \cite[Section 3]{MaMe12}.

  Let $\mc C$ be a full abelian subcategory of $\g\mod_{Z(\g)}$ that is invariant under projective functors. Following \cite[Section 3.3]{MaMe12}, a functor $F: \mc O_{\la+\Lambda}\rightarrow \mc C$ is said to  functorially commute with projective functors if for every $T\in {\mc Proj}$ there is an isomorphism $\kappa_T: T\circ F \rightarrow F\circ T$ such that for any natural transformation $\gamma: T\rightarrow T'$ of projective functors the following diagram is commutative:
 $$\xymatrixcolsep{3pc} \xymatrix{
 	T\circ F  \ar[r]^-{\gamma_{F}}  \ar@<-2pt>[d]_{\kappa_T} &   T'\circ F   \ar@<-2pt>[d]_{\kappa_{T'}} \\
 	F\circ T \ar[r]^{F(\gamma)} & F\circ T'}$$
 The following lemma  taken from \cite{CCM2}  is a consequence of \cite[Proposition 4]{Kh}.
 \begin{lem}{\em(}\cite[Lemma 1]{CCM2}{\em)} \label{lem::fpfO} Let $\la$ and $\mc C$ be as above. Suppose that $F_1, F_2:\mc O_{\la+\Lambda}\rightarrow \mc C$ are two exact functors that functorially commute with projective functors. Then $F_1\cong F_2$ if and only if $F_1(M(\la))\cong F_2(M(\la))$.
 \end{lem}

  \section{Blocks and simple objects in $\mc N(\zeta)$} \label{sect::WhCat} In this section, we assume that $\g$ is either basic classical or strange, unless mentioned otherwise. For a given central character $\chi$ of $\g$, we denote the central block in $\mc N(\zeta)$ corresponding to $\chi$ by $\mc N(\zeta)_{\chi}$, namely, it is the Serre subcategory of $\mc N(\zeta)$ consisting of objects annihilated by some powers of $\ker(\chi)$.
 We denote by $\mc N_0$ and $\mc N_0(\zeta)$ the analogous Whittaker categories of $\g_\oa$-modules. Similarly, we define the central block $\mc N_0(\zeta)_{\chi^0}$ in $\mc N_0(\zeta)$ for a given central character $\chi^0$ of $\g_\oa$.

    \subsection{The modules $\Gamma_\zeta(L(\mu))$} \label{Sect::31}
  For any $\la \in \h^\ast_\oa$, recall that we denote by $M_0(\la,\zeta)$ the   standard Whittaker module over $\g_\oa$ from \eqref{def::standardWhi} and by $L_0(\la,\zeta)$ the simple top of $M_0(\la,\zeta)$. Recall that $\Pi_\zeta$ denotes the set of simple roots in $\mf l_\zeta$.  Let $\Gamma_\zeta^0(-):\mc O^\oa\rightarrow \mc N_0(\zeta)$ denote the Backelin functor for $\g_\oa$-modules from \cite{B}. By \cite[Proposition 6.9]{B}, we have
  \begin{align*}
  	&\Gamma^0_\zeta(M_0(w\cdot\la))=M_0(\la,\zeta), \text{ for any $w\in W(\mf l_\zeta)$;}\\
  	&\Gamma^0_\zeta(L_0(\la))= \left\{\begin{array}{ll}
  		L_0(\la,\zeta), &  \text{~if $\langle \la,\alpha^\vee\rangle\not\in \Z_{\geq 0}$, for any $\alpha\in \Pi_\zeta$};\\
  		0, & \text{ otherwise}.
  	\end{array} \right.
  \end{align*} Generalizations of  Backelin's result to Lie superalgebras  can be found in \cite[Theorem 20]{Ch21} and \cite[Proposition 4, Theorem 6]{CC22}.

   By definition and the proof of \cite[Proposition 3]{AB}, we have
  \begin{align*}
  	&\Res ( \Gamma_\zeta(M)) \cong  \Gamma_\zeta^0(\Res(M)), \\
  	&\Ind ( \Gamma_\zeta^0(N)) \cong  \Gamma_\zeta(\Ind(N)),
  \end{align*} for any  $M\in \mc O$ and  $N\in \mc O^\oa$; see also the proofs of \cite[Theorem 20]{Ch21}.

     The simple module $L(\la)$ is said to be $\Pi_\zeta$-{\em free}   if every non-zero vector of $\g_\oa^{-\alpha}$ acts freely on $L(\la)$, for every $\alpha\in\Pi_\zeta$. In the case of integral weights, a classification of $\Pi_\zeta$-free simple modules is given in \cite[Section 4.2]{CCC}.  By \cite[Proposition 6.9]{B} it follows that
  	\begin{align*}
  		&\Gamma_\zeta(L(\la))\neq 0 \Leftrightarrow \Gamma_\zeta^\oa(\Res L(\la))\neq 0 \Leftrightarrow \text{ $L(\la)$ is $\Pi_\zeta$-free}.
  \end{align*}
\begin{lem} \label{lem::111} Let $\chi$ be a central character of $\g$.
	Let $S$ be a simple module in $\mc N(\zeta)_{\chi}$. Then $S$ is a quotient of $\Gamma_\zeta(L(\la))$, for some $\Pi_\zeta$-free  weight $\la \in \h_\oa^\ast$ such that $\chi_\la=\chi$.
\end{lem}
\begin{proof}
	Since $\Res S\in \mc N_0(\zeta)$, it follows that there is a $\g_\oa$-submodule of $\Res S$ isomorphic to $L_0(\mu,\zeta)$, for some $\mu\in \h^\ast_\oa$. By adjuction $S$ is a quotient of $\Ind L_0(\mu,\zeta)$, the latter is isomorphic to $\Gamma_\zeta(\Ind L_0(\mu))$. Since $\Gamma_\zeta: \mc O\rightarrow \mc N(\zeta)$ is exact, it follows that $S$ is a quotient of $\Gamma_\zeta(L(\la))$, for some composition factor $L(\la)$ of $\Ind L_0(\mu)$.
\end{proof}
Let  $\la\in \h_\oa^\ast$ be $\Pi_\zeta$-free. It is natural to ask whether $\Gamma_\zeta(L(\la))$  is always simple. Here is a list of known results: \begin{itemize}  		\item[$\bullet$] By \cite[Proposition 6.9]{B} and \cite[Theorem 20]{Ch21},  $\Gamma_\zeta(L(\la))$ is always simple provided that $\g$ is of type I.  In this case, $\Gamma_\zeta(L(\la))$ is isomorphic to the top of the standard Whittaker module $M(\la,\zeta):=K(M_0(\la,\zeta))$, which we denote by $L(\la,\zeta)$. 
	 \item[$\bullet$] By \cite[Theorem 6]{CC22},  $\Gamma_\zeta(L(\la))$   is   simple     if $\mf l_\zeta$ is a Levi subalgebra of $\g$ in the sense of \cite{Ma}. We refer to \cite[Section 3.3]{Ch21} for a discussion on the situations when $\mf l_\zeta$ is not a Levi subalgebra of $\g$.  Similarly, in this case $\Gamma_\zeta(L(\la))$ is isomorphic to the  top $L(\la,\zeta)$ of the standard Whittaker module $M(\la,\zeta)$ that is defined similarly to  \eqref{def::standardWhi}. \end{itemize}
In both cases above, a weight $\la\in \h^\ast_\oa$ is $\Pi_\zeta$-free if and only if $\la$ is an anti-dominant weight when restricted to $[\mf l_\zeta,\mf l_\zeta]$. Furthermore, we have
	\begin{align*}
		&L(\la,\zeta)\cong L(\mu,\zeta)\Leftrightarrow M(\la,\zeta)\cong M(\mu,\zeta)\Leftrightarrow  W(\mf l_\zeta)\cdot \la =W(\mf l_\zeta)\cdot\mu.
	\end{align*} We refer to \cite[Theorem 6]{Ch21} for more details.

 \subsection{Strongly typical blocks of $\mc N(\zeta)$} \label{sect::stblock}

 In this subsection, we assume that $\g$ is basic classical.

 \begin{prop} \label{Prop::eqvNz}  Suppose that  $\chi: Z(\g)\rightarrow \C$ is a strongly typical central character with a perfect mate $\chi^0$.
 The functors
 \begin{align}
 &\Res(-)_{\chi^0},~\Ind(-)_{\chi},
 \end{align} given in Lemma \ref{lem:gorelik}  restrict to mutually inverse equivalences between the central blocks $\mc N(\zeta)_{\chi}$ and $\mc N_0(\zeta)_{\chi^0}$.
 \end{prop}
\begin{proof}
 Since the functors $\Res(-),~\Ind(-),~(-)_{\chi^0}$ and $(-)_\chi$ are well-defined functors between $\mc N$ and $\mc N_0$ (c.f.~\cite[Lemma 3]{Ch21}), the conclusion follows from Lemma \ref{lem:gorelik}.
\end{proof}


\begin{prop}\label{lem:gorelikeq}
	Suppose that $\la$ is strongly typical  and  dominant such that  $\chi:=\chi_\la$. Then the isomorphism classes in the set $\{\Gamma_\zeta (L(w{\circ}\la))|~w\in W\}$ is an exhaustive list of simple modules in $\mc N(\zeta)_{\chi}$.
	
	 Furthermore, if we let  $\la'\in \h^\ast$ be such that $\chi^0 = \chi^0_{\la'}$ is a perfect mate of $\chi$ and  $M_0(\la') = \Res(M(\la))_{\chi^0}$, then for any $\mu=w\circ \la,~\mu'=w'\circ \la\in W\circ \la$ we have \begin{align*}
		&\Gamma_\zeta(L(\mu))\cong \Gamma_\zeta(L(\mu')) \Leftrightarrow W(\mf l_\zeta)w\cdot\la' = W(\mf l_\zeta)w'\cdot \la'.
		\end{align*}
\end{prop}
\begin{proof}
	The equivalences  $\Res(-)_{\chi^0},~\Ind(-)_{\chi}$ between $\g\mod_\chi$ and $\g_\oa\mod_{\chi^0}$ in Lemma \ref{lem:gorelik} restrict to equivalences between $\mc N(\zeta)_{\chi}$ and $\mc N_0(\zeta)_{\chi^0}$ by Proposition \ref{Prop::eqvNz}. 	We calculate
	\begin{align*}
		&\Gamma_\zeta(L(w\circ \la))\cong \Gamma_\zeta(\Ind L_0(w\cdot \la')_\chi)  \cong \left(\Ind \Gamma_\zeta^0L_0(w\cdot \la')\right)_\chi,
	\end{align*} for some $\g_\oa$-weight  $\la'\in \h^\ast$ such that $M_0(\la') = \Res(M(\la))_{\chi^0}$ with $\chi^0 = \chi^0_{\la'}$ (cf. \cite[Lemma 3.1]{Co16}).  Since $\Gamma_\zeta^0(L_0(w\cdot \la'))$ is a simple object in $\mc N_0(\zeta)_{\chi^0}$, we conclude by Lemma \ref{lem:gorelik} that $\left(\Ind \Gamma_\zeta^0L_0(w\cdot \la')\right)_\chi$ is simple and hence $\Gamma_\zeta(L(w\circ\la))$ is simple.  Finally, let $\mu=w\circ \la ,  \mu'=w'\circ\la$, for some $w,w'\in W$. Since $L(\mu)\cong \Ind(L_0(w\cdot \la'))_\chi$, $L(\mu')\cong \Ind(L_0(w'\cdot \la'))_\chi$ (see, e.g., \cite[Lemma 3.1]{Co16}), it follows that  \[\Gamma_\zeta(L(\mu))\cong \Gamma_\zeta(L(\mu')) \Leftrightarrow \Gamma_\zeta^0(L_0(w\cdot \la'))\cong \Gamma_\zeta^0(L_0(w'\cdot \la')), \] which is equivalent to $ W(\mf l_\zeta)w\cdot\la' = W(\mf l_\zeta)w'\cdot \la'$ by \cite[Proposition 2.1]{MS}.
\end{proof}

\begin{ex}
	Let $\g =\gl(m|n)$ or $\mf{osp}(2|2n)$ with  a typical weight $\la \in \h^\ast$. Then  we have  $K(L_0(\la))=L(\la)$.
	In this case, $K(L_0(\la,\zeta))\cong \Gamma_\zeta(L(\la))$ since $\Gamma_\zeta(-)$ commutes with   $K(-)$.
\end{ex}

\subsection{A decomposition of $\mc N(\zeta)$} \label{sect::33}
This section is devoted to a block decomposition of $\mc N$, for any quasi-reductive Lie superalgebras. The following lemma  is a consequence of \cite[Proposition 1.9]{Mc}; see also \cite[Section 5]{MS} and \cite[Proposition 2.2.2]{Bro}.
\begin{lem} \label{lem::1}
	Let $M \in \mc N(\zeta)$. Then $M$ is locally finite over $Z(\mf l_\zeta)$. In particular, we have  $$\Res_{\mf l_\zeta}^\g M=\bigoplus_{\la \in {\h^\ast_\oa}} M_{\chi^{\mf l_\zeta}_\la},$$
	where $M_{\chi^{\mf l_\zeta}_\la}$ is annihilated by some power of $\ker\chi^{\mf l_\zeta}_\la,$ {for $\la\in \h^\ast_\oa$.}
\end{lem}
 Let $\Z\Phi$ denotes the $\Z$-span of $\Phi$ in $\h_\oa^\ast$. We define an equivalence relation $\sim$ on $\h^\ast_\oa$ by
\begin{align}
	&\la\sim \mu ~\Leftrightarrow ~\la - w\cdot \mu\in \Z\Phi,~ \text{for some $w\in W(\mf l_\zeta)$}.\label{eq::sim}
\end{align}
Furthermore, we have
\begin{prop} \label{prop::3}
	Let $\la,\mu \in \h^\ast_\oa$. Let $S_\la$ and $S_\mu$ be  simple quotients  of $\Gamma_\zeta(L(\la))$ and $\Gamma_\zeta(L(\mu))$, respectively.  If $\Ext_{\mc N(\zeta)}^1(S_\la,S_\mu)\neq 0$ then we have  $\la\sim\mu$.
	
\end{prop}
\begin{proof}Let
	\begin{align*}
		0\rightarrow S_\mu \rightarrow E\rightarrow S_\la\rightarrow 0,
	\end{align*} be a non-trivial short exact sequence in $\mc N(\zeta)$. Let $L_0(\la',\zeta)$ and  $L_0(\mu',\zeta)$ be simple submodules of $\Res S_\la$ and $\Res S_\mu$, respectively. Then  $\mu'\sim \mu$, $\la'\sim \la$.
	
	By Lemma \ref{lem::1}, $E$ decomposes into $E=\bigoplus_{\la\in {\h^\ast_\oa}} E_{\chi_\la^{\mf l_\zeta}}$. Therefore we have  $E_{\chi^{\mf l_\zeta}_{\la'}}, E_{\chi^{\mf l_\zeta}_{\mu'}}\neq 0$ since $L_0(\la',\zeta)$ and  $L_0(\mu',\zeta)$  are quotients of $M_0(\la',\zeta)$ and  $M_0(\mu',\zeta)$, respectively.
	
	By our assumption  	we have $$(U(\g)E_{\chi^{\mf l_\zeta}_{\mu'}})\cap (U(\g)E_{\chi^{\mf l_\zeta}_{\la'}}) \neq 0.$$ Now, we recall a result of Kostant \cite{Ko75} that for any finite-dimensional $\mf l_\zeta$-module V, we have  $$V\otimes  E_{\chi^{\mf l_\zeta}_{\mu'}}\subseteq \bigoplus_{\gamma} E_{\chi^{\mf l_\zeta}_{\mu'+\gamma}},$$ where $\gamma$ is summed over all weights of $V$.
As $U(\g)$ is a direct sum of finite-dimensional $\mf l_\zeta$-modules  via the adjoint action, we have \begin{align*}
		&U(\g) E_{\chi^{\mf l_\zeta}_{\mu'}} \subseteq \bigoplus_{\gamma \in \Z\Phi}   E_{\chi^{\mf l_\zeta}_{\mu'+\gamma}}.
	\end{align*} Similarly, we have $U(\g) E_{\chi^{\mf l_\zeta}_{\la'}} \subseteq \bigoplus_{\gamma \in \Z\Phi}   E_{\chi^{\mf l_\zeta}_{\la'+\gamma}}.$ This implies that $\la'- w\mu'\in \Z\Phi$ for some $w\in W(\mf l_\zeta)$. Consequently, we have $\la \sim\mu$.
\end{proof}
 We have the following block decomposition of $\mc N(\zeta)$:
\begin{align}
	&\mc N(\zeta) = \bigoplus_{\ups\in {\h^\ast_\oa}/\sim} \mc N(\ups,\zeta), \label{eq::blockN}
\end{align}  where $\mc N(\ups,\zeta)$ is the Serre subcategory generated by composition factors of modules of the form $\Gamma_\zeta(L(\la))$, for some $\la \in \ups.$

For a given central character $\chi:Z(\g)\rightarrow \C$, recall that $\mc N(\zeta)_{\chi}$ denotes the corresponding central block.  Therefore we have the following decomposition  \begin{align}
	&\mc N =\bigoplus_{\zeta} \mc N(\zeta)  = \bigoplus_{\zeta,\chi} \mc N(\zeta)_\chi=  \bigoplus_{\zeta, \chi, \ups} \mc N(\ups,\zeta)_{\chi}, \label{eq::decomp}
\end{align} where $\mc N(\ups,\zeta)_{\chi} := \mc N(\ups,\zeta)\cap \mc N(\zeta)_\chi. $  In Proposition \ref{Thm::29} below we will show that all summands in \eqref{eq::decomp} are indecomposable in the case when $\g=\gl(m|n)$, for any $\ups\in \h^\ast/\sim$.

	\subsection{Cokernel categories $\mc O^{\nu\text{-pres}}_{\nu+\Lambda}$  and $\mc W(\zeta)$}  \label{Sect::34}
	In this section, we fix a dominant weight  $\nu \in \h^\ast_\oa$ be such that $$W(\mf l_\zeta) = \{w\in W|~w\cdot \nu =\nu\}.$$     Recall that  $\Lambda(\nu)$ denotes the set of weights $\mu\in \nu+\Lambda$ such that $L(\mu)$ is $\Pi_\zeta$-free. Recall that $\mc N(\zeta)_{\nu+\Lambda}$ denotes  the Serre subcategory of $\mc N$ generated by composition factors of $\Gamma_\zeta(L(\mu))$, for $\mu\in \nu+\Lambda$. By \eqref{eq::blockN} we have $\mc N(\zeta)_{\nu+\Lambda}=\bigoplus_{\ups\in\nu +\Lambda/\sim} \mc N(\ups,\zeta)$.

	For a given $\g$-module $M$, we let  $\coker(\mc F\otimes M)$ denote the {\em coker}-category of $M$, that is, $\coker(\mc F\otimes M)$ is the full subcategory of the category of all $\g$-modules  $N$ that have a presentation of the form
	\begin{align*}
		&X\rightarrow Y\rightarrow N\rightarrow 0,
	\end{align*}
	where $X$ and $Y$ are isomorphic to direct summands of $ E\otimes M$ for some finite-dimensional
	weight $\g$-modules $E$; see, e.g., \cite[Section 7.1]{CM21}.
	
	We define two cokernel categories of $\g$-modules as follows. First, we define
	\begin{align*}
	&\mc O_{\nu+\Lambda}^{\nu\text{-pres}}:=\coker(\mc F\otimes \Ind M_0(\nu)).
	\end{align*}
	 We refer to \cite[Lemma 20]{Ch212} and \cite[Section 4]{CCM2} for more details. In particular,  $\mc O_{\nu+\Lambda}^{\nu\text{-pres}}$ is the full subcategory of $\mc O_{\nu+\Lambda}$ consisting of modules $M$ that have  a two step presentation of the form $P_2\rightarrow P_1\rightarrow M\rightarrow 0,$
where $P_1,P_2$ are projective modules in $\mc O_{\nu+\Lambda}$ such that any simple quotient of $P_1$ or $P_2$ is of the form $L(\mu)$ with $\mu \in \Lambda(\nu)$. Furthermore, $\mc O^{\nu\text{-pres}}_{\nu+\Lambda}$ admits the structure of abelian category via an equivalence to  the Serre quotient category $\mc O_{\nu+\Lambda}/{\mc I}_{\nu}$ in the sense of \cite[Chapter III]{Gabriel}, where $\mc I_\nu$ is the Serre subcategory generated by all simple modules $L(\mu)$, where $\mu \in \nu+\Lambda$ with $\mu\notin \Lambda(\nu)$; see \cite[Lemma 12]{CCM2}.

Following \cite[Section 4.4.3]{CCM2}, we define a  full subcategory $\mc W_{\nu}(\zeta)$ of $\mc N(\zeta)_{\nu+\Lambda}$ as follows:
	\begin{align*}
	&\mc W_{\nu}(\zeta):=\coker(\mc F\otimes \Ind M_0(\nu,\zeta)).
\end{align*}
We remark that there is an intrinsic definition for $\mc W_{\nu}(\zeta)$ explained in \cite[Section 4.1]{Ch21}.

	Define a functor $F_\nu(-):\mc O_{\nu+\Lambda}\rightarrow \mc N(\zeta)$ as follows
\begin{align}
	&F_\nu(-):=\mc L(M_0(\nu),-)\otimes_{U(\g_\oa)} M_0(\nu,\zeta): \mc O_{\nu+\Lambda}\rightarrow \mc N(\zeta),\label{eq::defFv}
\end{align} where $\mc L(M_0(\nu),X)$ ($X\in \mc O_{\nu+\Lambda}$) is the maximal $(\g,\g_\oa)$-submodule of the $(\g,\g_\oa)$-bimodule $\Hom_\C(M_0(\nu),X)$ that is a direct sum of finite-dimensional $\g_\oa$-modules with respect to the adjoint action of $\g_\oa$; see, e.g., \cite[Section 4.3]{CCM2}. Then $F_\nu(-)$ is an exact functor and functorially commutes with projective functors by \cite[Theorem 26]{Ch212}. See also \cite[Section 7.3.1]{CCM2}, where the  case of Lie superalgebras of type I was considered.
	  	The following proposition is taken from \cite[Theorems 19, 26]{Ch212}  (see also    \cite[Section 7.1]{CC22}):
	\begin{prop} \label{prop::proF} For any quasi-reductive Lie superalgebra $\g$,
		the functor $F_\nu(-)$ restricts to an annihilator-preserving equivalence from $\mc O^{\nu\emph{-pres}}_{\nu+\Lambda}$ to $\mc W_\nu(\zeta)$. In particular, the following holds:
		\begin{itemize}
			\item[(1)] For any $\mu \in \nu+\Lambda$, we have   \begin{align}
				&F_\nu(L(\mu)) \text{ is simple } \Leftrightarrow F_\nu(L(\mu))\neq 0 \Leftrightarrow  \mu \in \Lambda(\nu). \label{eq::17}
			\end{align}
			\item[(2)] For any $\mu, \mu' \in \Lambda(\nu)$, we have \begin{align}
				& F_\nu(L(\mu))\cong  F_\nu(L(\mu')) \Leftrightarrow \mu =\mu'
			\end{align}
		\end{itemize}
	\end{prop}
	
	
	 The category $\mc O^{\nu\text{-pres}}_{\nu+\Lambda}$ is equivalent to $\mc O_{\nu+\Lambda}/\mc I_\nu$ (cf. \cite[Lemma 12]{CCM2}).	By \cite[Theorem 37]{CCM2}, the functor $F_\nu(-):\mc O_{\nu+\Lambda}\rightarrow \mc N(\zeta)$ satisfies  the universal property of the Serre quotient of $\mc O_{\nu+\Lambda}$ by the Serre subcategory  ${\mc I}_{\nu}$ in the sense of \cite[Corollaries III.1.2 and III 1.3]{Gabriel}, up to an equivalence  between the target category $\mc W_\nu(\zeta) = F_\nu(\mc O_{\nu+\Lambda})$ and $\mc O_{\nu+\Lambda}/\mc I_\nu$. That is, if $E(-):\mc O_{\nu+\Lambda}\rightarrow \mc N(\zeta)$ is an exact functor such that $E(L(\mu))=0$, for any $\mu\notin \Lambda(\nu)$, then there is a unique exact functor $E':{\mc W_\nu(\zeta)}\rightarrow \mc N(\zeta)$ such that  $E = E'\circ F_{\nu}$. 	The following is the main result of this subsection.
	\begin{thm} \label{thm::4}
 Let $\g$ be a basic classical Lie superalgebra.		The functors $\Gamma_\zeta(-),~F_\nu(-):\mc O_{\nu+\Lambda}\rightarrow \mc N(\zeta)$ are isomorphic and satisfy the universal property of the Serre quotient category of $\mc O_{\nu+\Lambda}$ by the Serre subcategory ${\mc I}_{\nu}$, up to an equivalence  between   $\mc W_\nu(\zeta)$  and $\mc O_{\nu+\Lambda}/\mc I_\nu$. In particular, for any $\mu\in \nu+\Lambda$ we have
		\begin{align*}
			&\Gamma_\zeta(L(\mu)) \text{ is simple } \Leftrightarrow \Gamma_\zeta(L(\mu))\neq 0\Leftrightarrow  \mu \in \Lambda(\nu).
		\end{align*}
		
		Furthermore, the set
		\begin{align*}
			&\{\Gamma_\zeta(L(\mu))|~\mu\in \Lambda(\nu)\}	
		\end{align*} is an exhaustive list of mutually non-isomorphic simple Whittaker modules in $\mc N(\zeta)_{\nu +\Lambda}$.
	\end{thm}
	

\begin{proof}
		We shall adapt the argument used in \cite[Corollary 38]{CCM2} and \cite[Corollary 29]{CC22}  to prove that  $\Gamma_\zeta(-),~F_\nu(-):\mc O_{\nu+\Lambda}\rightarrow \mc N(\zeta)$ are isomorphic.

First, we let $\la\in \nu+\Lambda$ be a generic weight; see Section \ref{sect::Goreq}. In particular, the corresponding central character $\chi_\la$ is strongly typical and has a perfect mate $\chi^0_{\la'}$, for some dominant weight $\la'$.  By Lemma \ref{lem:gorelikeq} and \cite[Lemma 3.1]{Co16}, it follows that
		$\Ind M_0(\la')_{\chi_{\la}} = M(\la)$.
{Thus, we have
\begin{align*}
\Gamma_\zeta(M(\la))\cong\Gamma_\zeta(\Ind M_0(\la'))\chi\cong&\Ind\Gamma_\zeta^0(M_0(\la'))_\chi\cong\\
&\Ind F_\nu^0(M_0(\la'))_\chi\cong F_\nu(\Ind M_0(\la'))_\chi\cong F_\nu(M(\la)).
\end{align*}}Here $F_\nu^0$ denotes the functor from \eqref{eq::defFv} between corresponding categories of  $\g_\oa$-modules. Since both $\Gamma_\zeta(-)$ and $F_\nu(-)$ functorially commute with projective functors, we have $\Gamma_\zeta(-)\cong F_\nu(-)$ as functors from $\mc O_{\la+\Lambda}=\mc O_{\nu+\Lambda}$ to $\mc N(\zeta)$ by Lemma \ref{lem::fpfO}.
		The theorem follows now from Proposition \ref{prop::proF}.
\end{proof}

  \subsection{Extension fullness}
  Let $\iota: \mc N\rightarrow \g\Mod$ be the natural inclusion functor.  Then $\iota$ induces a homomorphism of extension groups \begin{align}
  	&\iota^d_{M,N}: ~\Ext_{\mc N}^d(M, N) \rightarrow \Ext_{\mf g\Mod}^d(M, N),\label{eq::ExtFull}
  \end{align} for every $M, N\in {\mc N}$ and  $d\geq 0$. The following lemma is taken from \cite[Theorem 2]{CoM}.

  \begin{lem}[Coulembier-Mazorchuk] \label{lem::CM}
  	Suppose that $\g$ is a reductive Lie algebra. Then $\iota^d_{M, N}$ are isomorphisms,  for any $M, N\in {\mc N}$ and  $d\geq 0$. In particular, $\mc N$ is a Serre subcategory of $\g\Mod$.
  \end{lem}

  Let $M$ be a $\g$-module. By definition,  $M\in \mc N$ if and only if $\Res M\in \mc N_0$. We have the following useful consequence.
  \begin{lem} \label{cor::7}
  	$\mc N$ is a Serre subcategory of $\g\Mod$, for any quasi-reductive Lie superalgebra $\g$.
  \end{lem}

 Although the category  $\mc N(\zeta)$ does not have enough projective modules in general, the following corollary, generalizing Lemma 11,  shows that the extension groups in $\mc N(\zeta)$ can be computed in $\g\Mod$.

\begin{cor} For any quasi-reductive Lie superalgebra $\g$,
	we have isomorphisms
	\begin{align*}
		&\Ext^d_{\mc N}(M,N) \cong \Ext^d_{\g\Mod}(M,N),
	\end{align*} for any Whittaker modules $M,N$ in $\mc N$ and any integer $d\geq 0$.
\end{cor}
\begin{proof}   We are going to show that the homomorphism $\iota^d_{M,N}$ in \eqref{eq::ExtFull} is an isomorphism for any $M,N\in \mc N$  and $d\geq 0$.	Observe that each $M \in \mc N$ is a quotient of $\Ind \Res M$, where $\Res M\in \mc N_0$. By Lemma \ref{cor::7}, the $\mc N$ is a Serre subcategory of $\g\Mod$, and therefore we can apply \cite[Proposition 1]{CoM} for $\mc A=\mf g\Mod$, $\mc B=\mc N$ and $\mc B_0$ being the full subcategory of $\mc B$ consisting of all modules isomorphic to $\Ind V$, for some $V\in \mc N_0$.

	For any $V\in \mc N_0$, $M \in \mc N$ and   $d\geq 0$, we have
	\begin{align*}
		&\Ext_{\mc N}^d(\Ind V, M)
		\cong  \Ext_{\mc N_0}^d(V, \Res M)
		\cong \Ext_{\mf g_\oa\text{-Mod}}^d(V, \Res M) \cong
		\Ext_{{\mf g\text{-Mod}}}^d(\Ind V, M).
	\end{align*}
	The corollary now follows by applying \cite[Proposition 1]{CoM}.
\end{proof}

 \section{Principal finite $W$-superalgebras for basic classical Lie superalgebras} \label{sect::FWalg}

  In this section, we assume that $\g$ is basic classical. The goal of this section is to give a proof of Theorem B for basic classical Lie superalgebras.

  \subsection{Definitions} \label{sect::DefWal}
  The general definition of finite $W$-algebras for semisimple Lie algebras was  introduced by Premet in \cite{Pr02}, and it naturally extends to basic Lie superalgebras.   We review the basic notions about finite $W$-superalgebras; see, e.g., \cite{BR03,BBG, P13,PS13, P14, Zh14, ZS15,Wa,PSqn}.

  Recall that $\g$ admits an even non-degenerate invariant supersymmetric bilinear form $(\cdot|\cdot)$. Let $e\in \g_\oa$ be an even nilpotent element. Define $\chi\in \g^\ast$ by letting $\chi(x):=(e|x),$ for any $x\in \g$. By the Jacobson-Morozov theorem, $e$ can be included in an $\mf{sl}(2)$-triple $\langle e,h,f \rangle\subseteq \g_\oa$. The linear operator ${\rm ad}(h)$ given by the adjoint action of $h$ on $\g$ defines a {\em Dynkin} $\Z$-grading $\g =\bigoplus_{i\in \Z}\g(i)$, where $\g(i):= \{x\in \g|~[h,x] = ix\}$. The Dynkin grading is a {\em good grading} for $e$ in the sense of \cite{EK} (see also \cite{Ho12}).

  The map $\chi([\cdot,\cdot]):= (e|[\cdot,\cdot]): \g(-1)\times \g(-1)\rightarrow \C$ defines a non-degenerate bilinear form  on $\g(-1)$. Let $\mf l$ be a Lagrangian subspace with respect to this form, i.e., a maximal isotropic subspace of $\g(-1)$. Following \cite{Pr02}, we define the following nilpotent subalgebra
  $$\mf m:=\bigoplus_{i\leq -2} \g(i)\oplus \mf l.$$
  We denote by  $\zeta: \mf m \rightarrow \C$ the restriction of $\chi$ to $\mf m$, which  defines a one-dimensional representation $\C_\zeta$ of $\mf m$.   
 We define the {\em generalized Gelfand--Graev module}  $$Q_\zeta:=U(\g)/I_\zeta\cong U(\g)\otimes_{U(\mf m)}\C_\zeta,$$ where $I_\zeta$ is the left ideal generated by elements of the form $x-\zeta(x)$, for $x\in \mf m$. This module is also called generalized Whittaker module in \cite{PSqn}. The {\em finite $W$-superalgebra} $W_\zeta$ associated to the nilpotent element $e$ is defined as $$W_\zeta:=\End_{U(\g)}(Q_\zeta)^{op}.$$
 We denote by $W_\zeta^0$ the finite $W$-algebra associated to the even part $\g_\oa$ and the nilpotent element $e$ from \cite{Pr02}.

 Let $\pr: U(\g)\rightarrow U(\g)/I_\zeta$ denote the natural projection. As in the Lie algebra case, $W_\zeta$ can be identified with
 \begin{align}
 	&W_\zeta = \{\pr(y)\in Q_\zeta|~[x,y]\in I_\zeta, \text{for all }x\in \mf m\}.
 \end{align} The algebra structure on $W_\zeta$ is given by $\pr(y_1)\pr(y_2) = \pr(y_1y_2),$ for any $y_1, y_2\in U(\g)$ such that $[x,y_1], [x,y_2]\in I_\zeta$ for all $x\in \mf m$.



\subsection{Equivalence of categories}    In the remaining parts of the paper, we assume that  $e\in \g_\oa$ is a {\em principal} nilpotent element, that is, the kernel of ${\rm ad}e$ on $\g_\oa$ has minimal dimension (which is equal to the rank of $\g_\oa$). In this case, $\mf m_\oa$ is the nilpotent radical of a Borel subalgebra of $\g_\oa$, and thus $\zeta$ determines a character of the even nilpotent radical. We denote  the corresponding Whittaker categories of $\g$-modules by $\mc N(\zeta)$ and of $\g_\oa$-modules by $\mc N_0(\zeta)$, respectively.

\subsubsection{Skryabin equivalence} Denote by $ W_\zeta$-Mod  the category of all  $W_\zeta$-modules.  Let $\g$-Wmod$^\zeta$ be the full subcategory of $\g\Mod$ consisting of  $\g$-modules on which  $x - \zeta(x)$ acts locally nilpotently, for any $x\in \mf m$.

A {\em Whittaker vector} in a $\g$-module $M$ is a vector $v\in M$ satisfying $xv=\zeta(x)v$, for any $x\in \mf m$. We denote the subspace of  Whittaker vectors in a $\g$-module $M$ by
 \begin{align*}
 	&\Whz(M):=\{m\in M|~xm=\zeta(x)m,~\text{for all }x\in \mf m\}.
 \end{align*}
$\Whz$ defines a functor from $\g\text{-Wmod}^\zeta$ to $W_\zeta{\Mod}$, which we refer to as the {\em Whittaker functor}. We also define the Whittaker functor for the principal finite $W$-algebra $W_\zeta^0$ by $$\text{Wh}_\zeta^0(-): \g_\oa\text{-Wmod}^\zeta\rightarrow W^0_\zeta{\Mod},$$ where $\Whz^0(M):=\{m\in M|~xm=\zeta(x)m,~\text{for all }x\in \mf m_\oa\}.$



 \begin{lem}[Skryabin equivalence] \label{Skr}
 	The Whittaker functor $$\emph{Wh}_\zeta(-): \g\text{\em -Wmod}^\zeta\rightarrow   W_\zeta\Mod$$  is an equivalence with quasi-inverse $Q_\zeta\otimes_{  W_\zeta}(-)$.
 \end{lem}

 As observed in \cite[Remarks 3.10-3.11]{Zh14}, the proof of Lemma \ref{Skr} given in \cite{Skr} extends to basic classical Lie superalgebras; see also \cite[Theorem 2.17]{ZS15}, \cite[Theorem 4.1]{SW20}. We formulate a version in Theorem \ref{thm::eqvthmpn} that is also applicable to the periplectic Lie superalgebras.

 \subsubsection{Equivalence of $\mc N(\zeta)$ and $W_\zeta\text{-}\emph{fdmod}$}
 Recall that $\fdWmod$ denotes the full subcategory of $W_\zeta$-Mod consisting of finite-dimensional modules of $  W_\zeta$. The goal of this subsection is to prove that the Skryabin equivalence gives rise to an equivalence between $\mc N(\zeta)$ and  $\fdWmod$.



We note that the inclusion of vector spaces $$\Whz(\Gamma_\zeta(L(\la)))\subseteq \Whz^0(\Res\Gamma_\zeta(L(\la)))$$ shows that
the former are always finite-dimensional, since $\Gamma_\zeta^0(\Res L(\la))$ has finite length as a $\g_\oa$-module. The following lemma shows that the simple $W_\zeta$-modules are finite dimensional.
We point out that this fact has been established earlier by different methods; see, e.g., \cite[Theorem 7.2]{BBG}, \cite[Proposition 3.7]{PSqn} and \cite[Corollary 3.12]{SW20}.

\begin{lem} \label{lem2}
 The category $\mc N(\zeta)$ is a full subcategory of {\em $\g$-Wmod$^\zeta$}. Furthermore, it contains all simple objects in {\em $\g$-Wmod$^\zeta$}. In particular, simple quotients of $W_\zeta$-modules of the form  $\emph{Wh}_\zeta(\Gamma_\zeta(L(\la)))$ with $\Pi_\zeta$-free weights $\la$ constitute all simple modules of $W_\zeta$.
\end{lem}
\begin{proof}
	
First, we   show that every object $M$ in $\mc N(\zeta)$ lies in $\g$-Wmod$^\zeta$. Let $y\in \mf m_\ob$. We note that $\zeta(y) =0$ since $\zeta$ defines a one-dimensional module of $\mf m$. Since $[y,y]\in \mf m_\oa$, it follows that  $[y,y]-\zeta([y,y])$ acts locally nilpotently on $M$. We note that $\zeta([y,y]) = 2\zeta(y)\zeta(y) =0$. Therefore, $y^2 = \frac{1}{2}[y,y]$ acts locally nilpotently on $M$. This proves that $M$ is an object  in $\g$-Wmod$^\zeta$.

	Let $S$ be a simple object in  $\g$-Wmod$^\zeta$. By \cite[Proposition 1]{Ch21}, it follows that $S$ is locally finite over $Z(\g_\oa)$. Consequently, $S$ is an object of $\mc N(\zeta)$.
	
	Finally, let $V$ be an irreducible representation of $ W_\zeta$. Then $V\cong \Whz(S)$, for some simple object $S$ in  $\g$-Wmod$^\zeta$ by Skryabin's equivalence. By Lemmas \ref{lem::111}, $S$ is isomorphic to a quotient of  $\Gamma_\zeta(L(\la))$, for some $\Pi_\zeta$-free weight $\la \in \h^\ast$. This completes the proof.
\end{proof}

The following theorem is the main result in this subsection.
\begin{thm}  \label{thm::eqvthm}
	The Whittaker functor $\emph{Wh}_\zeta(-)$ restricts to an equivalence from $\mc N(\zeta)$ to ${W}_\zeta\emph{-fdmod}$ with inverse $Q_\zeta\otimes_{ W_\zeta}(-)$.
\end{thm}
\begin{proof} Let $\mc N(\zeta)'$ denote the full subcategory of $\g$-Wmod$^\zeta$ consisting of finite-length module $M$ such that every composition factor of $M$ lies in $\mc N(\zeta)$. By  Lemmas \ref{Skr} and \ref{lem2}, the Whittaker functor $\Whz(-)$ restricts to an equivalence  of $\mc N(\zeta)'$ and $\fdWmod$. To prove that $\mc N(\zeta)'=\mc N(\zeta)$, it remains to  prove that  every object $\mc N(\zeta)'$ lies in $\mc N(\zeta)$.
	
	
	
	Let $M$  be an object in $\mc N(\zeta)'$. We shall proceed by induction on the length of $M$. If  $M$ is simple, then the conclusion follows by Lemma \ref{lem2}. 
	Suppose that we have a short exact sequence in $\g\Mod$
	\begin{align}
		&0\rightarrow M_1\rightarrow M\rightarrow M_2\rightarrow 0,
	\end{align} where both $M_1,M_2$ are non-zero. Then $M_1,M_2$ are objects in $\mc N(\zeta)$ by induction. The conclusion follows by Lemma \ref{cor::7}.
\end{proof}

 \begin{rem} \label{rem::17}
  		The subalgebra $\mf m$ appearing in the definition of the finite $W$-superalgebra depends a priori on the choice of a good $\Z$-grading and the Lagrangian subspace $\mf l$. However, it is proved that these algebras  are all isomorphic in the case when $\g=\g_\oa$ \cite{GG}. The analogue of this statement for basic and queer Lie superalgebras
  		is known to hold under the assumption that the dimension of $\g(-1)_{\ob}$ is even; see \cite[Remark 3.11]{Zh14}.
  		

		In our setup with $e$ an even principal nilpotent element in a basic classical Lie superalgebra $\g$, the subalgebra $\mf m_\oa$ is the nilpotent radical of a Borel subalgebra of $\g_\oa$. We explain that our notation of representation categories of $W_\zeta$ is unambiguous, namely, both the categories  $W_\zeta\text{-fdmod}$ and $W_\zeta\text{-Mod}$ are independent of the choice of the even principal nilpotent elements, good $\Z$-gradings and the Lagrangian subspace $\mf l$, up to equivalences given by Whittaker functors, their quasi-inverses and inner automorphisms of $\g$.

		Let $\zeta:\mf m\rightarrow\C$ be the character as defined in Section \ref{sect::DefWal}. For a different even principal nilpotent element $e'$ with the associated subalgebra ${\mf m'}$ and the character  $\zeta': {\mf m'}\rightarrow \C$, there exists an inner automorphism of $\g$ that interchanges $\mf m_\oa$ and $\mf m_\oa'$. {Let $W'_{\zeta'}$ be the finite $W$-superalgebra corresponding to $\mf m'$ and $\zeta'$.} As a consequence of Theorem \ref{thm::eqvthm} and Proposition \ref{prop::equivNz}, we have $W_\zeta\text{-fdmod}\cong \mc N^{\mf b}(\zeta)\cong \mc N^{\mf b'}(\zeta')\cong {W'_{\zeta'}}\text{-fdmod}$, where $\mc N^{\mf b}(\zeta)$ and $\mc N^{\mf b'}(\zeta')$ are  the Whittaker categories with respect to the Borel subalgebras $\mf b$ and $\mf b'$ such that the radicals of $\mf b_\oa$ and $\mf b_\oa'$ are respectively $\mf m_\oa$ and $\mf m_\oa'$. An analogous statement holds for $W_\zeta\text{-Mod}$ by an analogous argument.
\end{rem}

By Theorem \ref{thm::eqvthm}, the extension groups in $\fdWmod$ can be computed in the category $\mc N(\zeta)$.  The following corollary compute extension groups between modules in $\fdWmod$ in terms of Whittaker modules in $\g\mod$.

\begin{cor}
Let $E,F$ be two finite-dimensional $W_\zeta$-modules. Then we have
\begin{align*}
	&\Ext^d_{\emph{W}_\zeta\emph{-fdmod}}(E,F)\cong     \Ext^d_{\g\Mod}(Q_\zeta\otimes_{  W_\zeta}E,Q_\zeta\otimes_{  W_\zeta}F),~ \text{{for $d\geq 0$}.}
\end{align*}
\end{cor}

\subsection{Examples} \label{Sect::43}
In this subsection, we identify simple $W_\zeta$-modules  for $\g=\mf{gl}(1|2)$, $\mf{osp}(1|2)$ and $\g=\mf{osp}(2|2)$ by computing Whittaker vector subspaces in simple Whittaker modules.

For given positive integers $m,n$, recall that the general linear Lie superalgebra $\mathfrak{gl}(m|n)$ has a realization as the space of $(m+n) \times (m+n)$ complex matrices
	\begin{align} \label{gllrealization}
		\left( \begin{array}{cc} A & B\\
			C & D\\
		\end{array} \right),
	\end{align}
	where $A,B,C$ and $D$ are $m\times m, m\times n, n\times m, n\times n$ matrices,
	respectively. The bracket of $\gl(m|n)$ is given by the super commutator. 		Let $E_{ij}$, for $1\leq i,j \leq m+n$ be the elementary matrix in $\mathfrak{gl}(m|n)$ with  $(i,j)$-entry equals to $1$ and all other entries equal to $0$. Let $\str:\mf{gl}(m|n)\rightarrow \C$ denote the super-trace form.
	
	{In each of the following examples, we shall choose a Borel subalgebra $\mf b$ with a Cartan subalgebra $\h$, and we let $\mc O$  be the BGG category with respect to $\mf b$. For any given $\la \in \h^\ast$, we recall that  $L_0(\la,\zeta)$ and $M_0(\la,\zeta)$ denote the Kostant's  simple Whittaker $\g_\oa$-module from \cite{Ko78} and the standard Whittaker $\g_\oa$-module from \cite{MS}. Since we will set  the character $\zeta$ to be  non-singular,  it follows that  $M_0(\la,\zeta)=L_0(\la,\zeta)$, for any $\la \in \h^\ast_\oa$. Also, we recall that $M(\la)$ and $L(\la)$ denote the Verma module and its simple quotient, respectively.}

\subsubsection{Example: $\g =\mf{gl}(1|2)$ }\label{Sect::431}
\begin{ex}
	Suppose that $\g =\mf{gl}(1|2)$ with $\zeta\neq 0$.   {Let $\h$ be the standard Cartan subalgebra spanned by $E_{ii}$, for $1\leq i\leq 3$. Set $\mf b$ to be the standard Borel subalgebra spanned by $E_{ij}$, for $1\leq i\leq j\leq 3$.}
	
Consider the $\mf{sl}(2)$-triple $\langle e,h,f\rangle$ given by \begin{align*}
	&e:=E_{32},~h:=-E_{22}+E_{33},\text{ and }f:=E_{23},
	\end{align*}
 Consider the even non-degenerate invariant supersymmetric bilinear form $(\cdot|\cdot)$ determined by $(x|y):= \str(xy)$, for any $x,y\in \g$. Using the $\Z$-grading induced by the adjoint action of $h$,  we see that $\g(-1)=\C E_{21}+\C E_{13}$. We choose $\mf l=\C E_{13}$ so that the subalgebra  $\mf m=\C E_{23} +\C E_{13}$. {Note that $\mf m_\oa$ is the nilpotent radical of $\mf b_\oa$.} The character $\zeta$ on $\mf m$ is determined by $\zeta(E_{23})=1$ and  $\zeta(E_{13})=0.$
	
	We recall the explicit construction of simple Whittaker modules in $\mc N(\zeta)$ from \cite[Section 5.3]{Ch21}. The standard Whittaker module is defined as  ${M}(\la,\zeta):=K(M_0(\la,\zeta))= \Ind_{\g_{\geq 0}}^{\g}M_0(\la ,\zeta)$, for any $\la\in \h^\ast$. By \cite[Theorem 9]{Ch21}, every simple module in $\mc N(\zeta)$ is a quotient ${M}(\la, \zeta)$, which  we denote by $L(\la,\zeta)$.
	Then $M_0(\la, \zeta)=L_0(\la,\zeta)$ can be regarded as a submodule of $\Res {M}(\la,\zeta)$. Let $v\in\Whoa( M_0(\la,\zeta))$ be a non-zero Whittaker vector. The set
	\begin{align*}
		&\{v_1:=v,~v_2:=E_{21}v,~v_3:=E_{21}E_{31}v,~v_4:=2E_{31}v-E_{21}hv\},
	\end{align*} forms a basis for $\Whz^0(\Res  {M}(\la, \zeta))$ by \cite[Lemma 5.6]{BCW}. {Define the dual basis $\{\vare_1, \vare_2, \vare_3\}\subset \h^\ast$ by $\vare_i(E_{jj})=\delta_{ij}$. Write $\la =\la_1\vare_1+\la_2\vare_2+\la_3\vare_3$.} Let $c:=\la_1+\frac{1}{2}(\la_2+\la_3)$ and define	\begin{align*}
		&w= \left\{\begin{array}{ll}
			v_2+\frac{1}{2(1-c)}v_4, &  \text{~for } c\neq 1;\\
			v_4, & \text{for } c=1;
		\end{array} \right.
	\end{align*} By a direct computation, the set $\{v_1, w\}$ forms a basis for $\Whz(M(\la, \zeta))$. By \cite[Proposition 22]{Ch21}, we have
	\begin{align*}
		&\Whz(L(\la, \zeta)) = \left\{\begin{array}{ll}
			\C v\oplus \C w, &  \text{~for } \la \text{ typical}; \\
			\C v \oplus \C w/\C w, & \text{~for } \la \text{ atypical}.
		\end{array} \right.
	\end{align*}
	We remark that the construction of simple objects  $\Whz(L(\la,\zeta))$ for general $\gl(m|n)$ have been studied in  \cite[Section 7]{BBG} via the triangular decomposition of $W_\zeta$ arising from the shifted Yangian of type A.
\end{ex}

\subsubsection{Example: $\g=\mf{osp}(1|2)$}
Recall that $\mf{osp}(1|2)\subseteq \mf{gl}(1|2)$ has the following generators
\begin{equation*}
	e=\left( \begin{array}{ccc} 0 & 0 &0\\
		0& 0 & 1\\
		0& 0 & 0
	\end{array} \right)
	, \hskip 0.5cm
	f=\left( \begin{array}{ccc} 0 & 0 &0\\
		0& 0 & 0\\
		0& 1 & 0
	\end{array} \right)
	, \hskip 0.5cm
	h=\left( \begin{array}{ccc} 0 & 0 &0\\
		0& 1 & 0\\
		0& 0 & -1
	\end{array} \right),
\end{equation*}
\begin{equation*}
	F=\left( \begin{array}{ccc} 0 & 1 &0\\
		0& 0 & 0\\
		-1& 0 & 0
	\end{array} \right),
	\hskip 0.5cm
	E=\left( \begin{array}{ccc} 0 & 0 &1\\
		1& 0 & 0\\
		0& 0 & 0
	\end{array} \right).
\end{equation*}

Consider the principal nilpotent element $e$ inside $\langle e,h,f\rangle \cong \mf{sl}(2)$. Then the $\Z$-grading 	$\g = \bigoplus_{-2\leq k\leq 2}\g(i)$ induced by $h$ is given by
\begin{align*}
	&\g(-2) =\C f, ~\g(-1) = \C F,~\g(0)=\mf h:= \C h,~\g(1)= \C E,~\g(2) = \C e.
\end{align*}
Consider the even non-degenerate invariant supersymmetric bilinear form $(\cdot|\cdot)$ determined by $(x|y):= -\str(xy)$, for any $x,y\in \g$. Then the corresponding Lagrangian subspace is zero. We have $\mf m=\C f$ with  the character $\zeta:\mf m \rightarrow \C$ determined by $\zeta(f) = 1$. This gives rise to the corresponding principal finite $W$-superalgebra $W_\zeta$; see also \cite[Section 3.1]{P14}.

Consider the Borel subalgebra $\mf b:=\C f+\C F+\mf h$.  Let $\{\delta\}$ be the dual basis for $\mf h^\ast$ determined by $\delta(h)=1$.    Define  $L(\la,\zeta):=\Gamma_\zeta(L(\la))$ in $\mc N(\zeta)$, for any $\la\in \h^\ast.$

We have  $$L(\la,\zeta)\neq 0\Leftrightarrow L(\la) \text{ is a $U(e)$-free module}\Leftrightarrow \la(h)\not\in \Z_{\leq 0}.$$ In this case, we have $M(\la) = L(\la)$, and therefore $L(\la,\zeta)$ is by definition a $\g$-submodule of the completion $\ov{M(\la)}$. Note that $\ov{M(\la)} = \ov{U(\g_\oa)v} \oplus E(\ov{U(\g_\oa)v})$ as a vector space, where $v\in M(\la)$ is a non-zero highest weight vector.  We use the notation  $\sum_{t\geq 0, s=0,1} c_{s,t} E^se^t v$, for $c_{st}\in \C$, to denote   elements of $\ov{M(\la)}$.

Suppose that $L(\la,\zeta)\neq 0$ (i.e., $\la(h)\not\in \Z_{\leq 0}$). We are going to give a construction of $L(\la,\zeta)$ as follows.
 Set
\begin{align*}
	&w:=\sum_{k\geq 0} a_k e^k v \in \ov{M(\la)},
\end{align*}
where $a_0:=1$ and $a_1,a_2,\ldots$ are inductively determined by
\begin{align*}
	& a_{k+1}:=\frac{- a_k}{(k+1)(k+\la(h))}, \text{ for }k\geq 0.
\end{align*}
\begin{cor} The following set $$\{L(\la,\zeta)|~\la(h)\notin \Z_{\leq 0}\}$$ is an exhaustive list of mutually non-isomorphic simple objects in $\mc N(\zeta)$. Furthermore,  we have the following description of   Whittaker vectors in these simple  modules
	\begin{align}
		&\emph{Wh}_\zeta(L(\la,\zeta)) = \C w \oplus \C Fw. \label{eq::419} \end{align}
	In particular, $L(\la,\zeta)$ is generated by  $w, Fw$.
\end{cor}
\begin{proof}
	Set $X:=\{\la\in \h^\ast|~\la(h)\notin\Z_{\leq 0}\}$. We first show that $L(\la,\zeta)$ is simple, for any $\la\in X$. To see this, we note that $\Res L(\la,\zeta) = \Res \Gamma_\zeta M(\la)$ has a composition series of length two. If $L(\la,\zeta)$ were not simple, then it would contain a simple submodule $L$ such that $\Res L$ is a simple Whittaker module over $\g_\oa$. Hence, $h$ acts on $L$ injectively and $EL = FL =0$, which implies that $[E,F]L=hL=0$, a contradiction. This proves the simplicity of $L(\la,\zeta)$.
	
	Recall that $L_0(\la,\zeta)\not\cong L_0(\mu,\zeta)$ for any  different $\la,\mu\in X$. Since $\Res L(\la,\zeta)$ has  composition factors $\{L_0(\la,\zeta), L_0(\la+\delta, \zeta)\}$ for any $\la\in X$, it follows that $ L(\la,\zeta)\not\cong  L(\mu,\zeta)$ for any $\la,\mu \in X$ with $\la\neq \mu$.
	
	Note that  $\dim \Whz(L(\la,\zeta)) = \dim \Whz^0(\Res L(\la,\zeta))=2$.
	We calculate
	\begin{align*}
		&fw = \sum_{k\geq 0} a_k [f,e^k] v =  \sum_{k\geq 0} -a_{k+1}(k+1)e^k(k+h) v  = \sum_{k\geq 0} a_k e^kv =  w, \\
		&fFw = Ffw =  Fw.
	\end{align*}
	Therefore, $w,Fw\in \Whz(L(\la,\zeta))$.  It remains to show that $w,Fw$ are linearly independent. For this, it is enough to show that $Fw\not=0$. But if $Fw=0$, then $2w=2fw=F^2w=0$, which is a contradiction.
This completes the proof of \eqref{eq::419}.
\end{proof}

\subsubsection{Example: $\g =\mf{osp}(2|2)$ }

 In this example, 	we construct   Whittaker vector subspaces of simple Whittaker modules for $\g :=\mf{osp}(2|2)\cong \mf{sl}(1|2)$, for non-singular $\zeta$. The matrix realization of the orthosymplectic Lie superalgebra $\mf{osp}(2|2)$ inside $\mf{gl}(2|2)$ is given by
		\begin{equation}
		\g=
			\left\{ \left( \begin{array}{cccc} d &0 & x &y\\
				0 & -d& v & u\\
				u& y & a &b\\
			-v &-x & c& -a \\
			\end{array} \right):
			\begin{array}{c}
				\,\, x,y,v,u\in \C;\\
				a,b,c,d\in \C;
			\end{array}
			\right\}. \label{eq::osp}
		\end{equation}
		We introduce the following generators of $\g_\oa$
		\begin{align*}
		&e:= E_{34},~h:=E_{33}-E_{44},~f:=E_{43}, \text{ and } h':=E_{11}-E_{22}.
		\end{align*}
Also, we set the following generators of $\g_\ob$
	\begin{equation*}
		X=\left( \begin{array}{cccc} 0 &0 & 1 &0\\
			0 & 0& 0 & 0\\
			0& 0 & 0 &0\\
			0 &-1 & 0& 0 \\
		\end{array} \right)
		, \hskip 0.5cm
		Y=\left( \begin{array}{cccc} 0 &0 & 0 &1\\
			0 & 0& 0 & 0\\
			0& 1 & 0 &0\\
			0 &0 & 0& 0 \\
		\end{array} \right)
	\end{equation*}
\begin{equation*}
		U=  \left( \begin{array}{cccc} 0 &0 & 0 &0\\
			0 & 0& 0 & 1\\
			1& 0 & 0 &0\\
			0 &0 & 0& 0 \\
		\end{array} \right)
		,\hskip 0.5cm
		V=\left( \begin{array}{cccc} 0 &0 & 0 &0\\
			0 & 0& 1 & 0\\
			0& 0 & 0 &0\\
			-1 &0 & 0& 0 \\
		\end{array} \right).		
	\end{equation*}
Note that $\g$ is of type-I and we have the following $\Z$-gradation for $\g$: $\g=\g_{-1}+\g_0+\g_{+1}$, where $\g_{-1}=\C X \oplus\C Y$, $\g_0=\C e \oplus\C f \oplus\C h \oplus \C h'$, and $\g_{+1}=\C U \oplus\C V$.
 Set $\g_{\leq 0}:=\g_0+\g_{-1}$.
	
		Consider the even non-degenerate invariant supersymmetric bilinear form $(\cdot|\cdot)$ determined by $(x|y):= -\str(xy)$, for any $x,y\in \g$. Consider  the even principal nilpotent element $e$ inside  $\langle e,h,f\rangle\cong \mf{sl}(2)$.
 Then the $\Z$-grading
	$\g = \bigoplus_{-2\leq k\leq 2}\g(i)$ induced by the adjoint action of $h$ is as follows:
    \begin{align*}
    &\g(-2) = \C f, ~\g(-1) = \C X\oplus \C V,\\
    &\g(0) = \mf h:= \C h'\oplus \C h,\\ &\g(2) = \C e,~\g(1) = \C Y\oplus \C U.
    \end{align*}
	
	Following \cite[Section 3.3]{P14}, we pick the Lagrangian subspace $\mf l = \C X\subseteq \g(-1)$, and thus $\mf m = \C f \oplus \C X$. The corresponding character $\zeta:\mf m \rightarrow \C$ is determined by $\zeta(f)= 1,~\zeta(X) =0.$ Let $W_\zeta$ denote the corresponding principal finite $W$-superalgebra.
	

  Consider the Borel subalgebra $\mf b:=\C f+\mf h +\g_{-1}$ of $\g$.  Let $\{\vare, \delta\}$ be the dual basis for $\mf h^\ast$ with respect to the ordered basis $\{h',h\}$ for $\h$. For a given weight $\la =\la_1\vare +\la_2\delta \in \h^\ast$,  define  $L(\la,\zeta):=\Gamma_\zeta(L(\la)) \in \mc N(\zeta)$, which is the simple quotient of the standard Whittaker module $M(\la,\zeta):= \Ind_{\g_{\leq 0}}^{\g}L_0(\la ,\zeta)$ by \cite[Theorem 20]{Ch21} whenever it is non-zero. Regard $L_0(\la,\zeta)$ as a $\g_\oa$-submodule of $\Res M(\la,\zeta)$. We note that the Casimir operator $h^2+2h+4fe = h^2-2h+4ef$ acts on $L_0(\la,\zeta)$ as the scalar $\gamma:= (\la_2-1)^2-1$.

  Let $\langle \cdot, \cdot \rangle:\h^\ast\times \mf h^\ast\rightarrow \C$  be the symmetric bilinear form determined by $\langle \vare,\vare\rangle =1 = -\langle \delta,\delta\rangle$ and $\langle \delta,\vare\rangle = 0$.  Then $\la$ is typical with respect to $\mf b$ if and only if $$\langle \la+\rho, \alpha\rangle \neq 0,~\text{ for all  roots }\alpha = \pm \vare \pm \delta,$$
where  $\rho = -\vare-\delta$ is the corresponding Weyl vector. Equivalently,  $\la$ is typical if and only if $(\la_1 -\la_2)(\la_1 +\la_2-2)\neq 0$.   By \cite[Corollary 6.8]{CM21}, we have
\begin{align}
&M(\la,\zeta) = L(\la,\zeta)\Leftrightarrow \text{ $\la$ is typical}. \label{eq::414}
\end{align}
Alternatively, one can also deduce \eqref{eq::414} from \cite{Ch21} by applying the Backelin functor $\Gamma_\zeta$.   The following   gives a classification of  irreducible representations of $ W_\zeta$ for $\g=\mf{osp}(2|2)$:
	\begin{cor}
		 Let $v$ be a non-zero Whittaker vector of $L_0(\la,\zeta)$ and define $w$ in $M(\la,\zeta) = \wedge(\g_{+1})\otimes {L_0}(\la,\zeta)$ to be $$w:=(\la_1-2)Vv +2Uv +Vhv.$$ Then $\{v,w\}$ forms a basis for   $\emph{ Wh}_\zeta(M(\la,\zeta))$.  Furthermore, we have the following exhaustive list of simple modules of $ W_\zeta$:
		\begin{align*}
			&\emph{Wh}_\zeta(L(\la, \zeta)) = \left\{\begin{array}{ll}
				\C v\oplus \C w, &  \text{~for } \la \text{ typical,} \\
				\C v \oplus \C w/\C w, & \text{~for } \la \text{ atypical,}
			\end{array} \right.
		\end{align*} where $\la\in \h^\ast$ runs over all anti-dominant weights.
	\end{cor}
\begin{proof}
	 We calculate the basis for $\Whz(M(\la,\zeta))$ as follows. By \cite[Lemma 2]{Mc2}, the restriction $\Res M(\la,\zeta) \cong \Lambda(\g_{+1})\otimes L_0(\la,\zeta)$ has a basis
	\begin{align*}
		&\{V^{i}U^j h^kv|~i,j =0,1, k \in \N\}.
	\end{align*} By \cite[Theorem 4.6]{Ko75} (see also \cite[Lemma 5.12]{MS}), we have $\dim(\Whz^0(M(\la,\zeta))) = 4.$  By a direct computation, the following set forms a basis for $\Whz^0(M(\la,\zeta))$:
	\begin{align*}
		&\{v,~Vv,~VUv,~2Uv+Vhv\}.
	\end{align*}
	Let us calculate the action of $X$ on these vectors:
	\begin{align*}
		&Xv=0,\\
		&XVv = -2 v,\\
		&XVUv = 2Vv -2Uv -Vhv -\la_1 Vv,\\
		&X(2Uv+Vhv) = 2\la_1 v -4v.
	\end{align*}
	It follows that $\Whz(M(\la,\zeta))$ has a basis $\{v,~(\la_1-2)Vv +2Uv +Vhv\}$.

	By \eqref{eq::414} it remains to consider the case of atypical weight $\la$. In this case, $\Whz(L(\la,\zeta))$ is an indecomposable module of length two over $ W_\zeta$. We calculate
	\begin{align*}
	&YVv = (h-h')v,~YUv=\frac{1}{2}(\gamma - h^2 +2h)w,~YVhw = (h-h')hw.
	\end{align*}
It follows that $Yw = ((2-\la_1)\la_1 +\gamma)v =  { ( (2-\la_1)\la_1 + \la_2(\la_2-2))v} = 0$ since $\la$ is atypical. Therefore, we have $\g_{-1}w=0$.

Recall that $h'$ acts on $L_0(\la,\zeta)$ as the scalar $\la_1$. Since $h' w= (\la_1+1)w$ and $U(\g)w = U(\g_{\geq 0}) w,$ it follows that $U(\g)w$ is a proper submodule   of $M(\la,\zeta)$. Applying $\Whz(-)$ to the following short exact sequence
\begin{align*}
&0\rightarrow U(\g)w\rightarrow M(\la,\zeta) \rightarrow M(\la,\zeta)/U(\g)w = L(\la,\zeta)\rightarrow 0,
\end{align*} the conclusion follows.
\end{proof}

\section{Principal finite $W$-superalgebra of periplectic Lie superalgebra $\pn$} \label{sect::pn}

\subsection{Skryabin type equivalence} \label{sect::skr}
 The goal of this section is to formulate a  Skryabin type equivalence for a principal finite $W$-superalgebra arising from an even $\Z$-grading. To explain this result,  let $\g$ be an arbitrary quasi-reductive Lie superalgebra  with an $\mf{sl}(2)$-triple $\langle e,h,f\rangle \subseteq \mf g_\oa$ such that $e$ is   principal  in $\g_\oa$. Suppose that the adjoint action of $h$ on $\g$ gives rise to an even $\Z$-grading $\g=\bigoplus_{k\in 2\Z}\g(k).$ Define the following nilpotent subalgebra of $\g$ \begin{align*}
	&\mf m:=\bigoplus_{k\leq -2}\g(k).
	\end{align*}
	
	Let $(\cdot|\cdot)_0$ be a  non-degenerate  invariant symmetric bilinear form of $\g_\oa$. Suppose that the non-singular character $\zeta: \mf m_\oa \rightarrow \C$ determined by $\zeta(x):=(x|e)_0,$ for $x\in \mf m_\oa$,  extends to a character of $\mf m$. We define the corresponding principal finite $W$-superalgebra $W_\zeta$   in a similar fashion, namely,
	\begin{align}
		&W_\zeta:=\End_{U(\g)}(Q_\zeta)^{\text{op}},\label{def::generalW}
	\end{align} where $Q_\zeta:=U(\g)/I_\zeta$ is the Gelfand-Graev type module with the left ideal $I_\zeta$ of $U(\g)$   generated by elements $x-\zeta(x)$, for $x\in \mf m$.
	
	  We retain the notations used in the previous sections. Namely, $\g$-Wmod$^\zeta$  denotes the category of all  $\g$-modules on which $x-\zeta(x)$ acts locally nilpotently, for all $x\in\mf m.$ Also, $W_\zeta\Mod$ and $  W_\zeta\text{-fdmod}$ denote the category of all $W_\zeta$-modules and finite-dimensional $  W_\zeta$-modules, respectively. Again, we define the Whittaker functor $\Whz(-)$ from $\g\text{-Wmod}^\zeta$ to $W_\zeta\Mod$ by  \begin{align*}
	&\Whz(M):=\{v\in M|~xv=\zeta(x)v,~\text{for all }x\in \mf m\},
\end{align*} for $M\in \g$-Wmod$^\zeta$. Since $\mf m_\oa$ is the nilradical of {the underlying even subalgebra of} a Borel subalgebra of {$\g$}, we can define the corresponding Whittaker category $\mc N(\zeta)$ {and BGG category $\mc O$.}

\subsubsection{A necessary and sufficient condition} \label{sect::511}
In this subsection, fix $0\leq m'\leq m$ and let
\begin{align}
	&\{u_1,\ldots, u_{m'}\}\subseteq \mf m_\oa,~\{u_{m'+1},\ldots, u_m\}\subseteq \mf m_\ob,\label{eq::55} \\  &\{x_1,\ldots , x_{m'}\}\subseteq\g_\oa, ~\{x_{m'+1},\ldots, x_m\}\subseteq\g_\ob,
\end{align} be homogenous elements with respect to both $\Z_2$- and $\Z$-gradings such that $x_s\in \g(-2+d_s)$, where $d_s>0$, for any $1\leq s\leq m$. Define the following notations on  elements ${\bf a} = (a_1,\ldots,a_m)$ in $X:=\mathbb Z_{\geq 0}^{m'}\times \{0,1\}^{m-m'}$ and in $U(\g)$ by letting
\begin{align*}
	&|{\bf a}| = \sum_{s} a_s,~\wt{\bf a }= \sum_{s}d_sa_s,\\
	&x^{\bf a} = x_1^{a_1}x_2^{ a_2}\cdots x_m^{ a_m}, \\
	&u^{\bf a} = (u_1-\zeta(u_1))^{a_1}(u_2-\zeta(u_2))^{a_2}\cdots (u_{m'}-\zeta(u_{m'}))^{a_{m'}} u_{m'+1}^{a_{m'+1}}\cdots u_{m}^{a_m}.
\end{align*}
Consider any linear ordering $<$ on $X$ subject to the condition \begin{align*}
	&{\bf a}< {\bf b} \text{ whenever either }\wt {\bf a} <\wt {\bf b}  \text{ or } \wt {\bf a} = \wt {\bf b}, |{\bf a}| >|{\bf b}|.
\end{align*}
The following is the main result in this subsection.
\begin{thm} \label{thm::eqvthmpn} Retain the notations above.
	Suppose that the following conditions are satisfied:
	\begin{itemize}
		\item[(1)] $\{u_i|~1\leq i\leq m\}$ is a basis for $\mf m$.
		\item[(2)] For any $i$, we have  $[u_i,x_i]\in \g(-2)\cap \g_\oa$ and $\zeta([u_i,x_i]) =1$.
		\item[(3)] For any $i\neq j$. If  $[u_i, x_j]\in \g(-2)$  then $\zeta([u_i,x_j])=0$.
	\end{itemize}
	Then, for any $M\in \g\mod$ with $v\in \emph{Wh}_\zeta(M)$, we have \begin{align}
		&u^{\bf a} x^{\bf a} v =cv,\text{ for some non-zero  scalar $c$;}\label{eq::66}\\
		&u^{\bf a}x^{\bf b} v= 0, \text{ when }{\bf a} > {\bf b}. \label{eq::67}
	\end{align}
	 In addition, $\{x^{\bf a}{1_\zeta}|~ {\bf a}\in X\}$ forms a basis for $Q_\zeta$ as a  free right $W_\zeta$-module, {where $1_\zeta=\pr(1_{U(\g)})$ is the image of $1_{U(\g)}\in U(\g)$ in $Q_\zeta$.}
	In this case, the Whittaker functor $\emph{Wh}_\zeta(-)$ is an equivalence from $\g$\emph{-Wmod}$^\zeta$  to $  W_\zeta\emph{-Mod}$ with inverse $Q_\zeta\otimes_{  W_\zeta}(-)$.
	The target category of $\mc N(\zeta)$ under $\emph{Wh}_\zeta(-)$ is $  W_\zeta$-$\emph{fdmod}$.
\end{thm}
	Before giving a proof,  we illustrate Conditions (1)--(3) of Theorem \ref{thm::eqvthmpn} with some examples.
	\begin{ex}
		Let $\g$ be a basic classical Lie superalgebra with the algebra $\mf m$ and the character $\zeta$ arising from Section \ref{sect::DefWal} such that $\mf m =\bigoplus_{k\leq -2}\g(k)$.
		Fix a homogenous basis $\{u_i\}_{i=1}^{m}\subset \mf m$ as in \eqref{eq::55}. Recall the even non-degenerate invariant supersymmetric bilinear form $(\cdot|\cdot)$ on $\g$ from Section \ref{sect::DefWal}. For any $i,j\in \Z$,  we may note that $(\cdot|\cdot)$ forms a non-degenerate paring   between $\g(i)$, $\g(-i)$,  and $(\g(i)|\g(j))=0$ unless $j=-i$. Since elements $\{[u_i,e]\}_{i=1}^m$ are linearly independent, there exist   homogenous elements $x_1, x_2, \ldots, x_m$ in $\g$ such that
		\begin{align*}
			&(x_j|[u_i,e]) = ([x_j,u_i]|e) = \delta_{ij}, \text{ for any }i,j,
		\end{align*}  and $[x_j,u_j]\in \g(-2)$, for any $1\leq j\leq m$.
\end{ex}

 \begin{ex} \label{eg::25}
			Recall the matrix realization of $\gl(n|n)$ from  \eqref{gllrealization}. The queer Lie superalgebra $\mf q(n)$ can be realized as the following subalgebra of $\gl(n|n)$:
		\[\mf g:=
		\mf q(n)=
		\left\{ \left( \begin{array}{cc} A & B\\
			B & A\\
		\end{array} \right) \right\},
		\] where $A,B  \in \C^{n\times n}$.   The principal finite $W$-superalgebra for $\mf q(n)$ has been defined in  \cite[Section 3]{Zh14} via an odd degenerate invariant supersymmetric bilinear form. 	 We shall illustrate Theorem \ref{thm::eqvthmpn} with this case. {We point out the fact that   $\Whz(-):\g\text{-Wmod}^\zeta\rightarrow W_\zeta\text{-Mod}$ is an equivalence for this case has been established in \cite[Theorem 3.9]{Zh14}.}
		
	 For $1\leq a,b\leq n$, recall that we let $E_{ab}\in \C^{n\times n}$ be the elementary matrix with $1$ at the $(a,b)$-position and $0$ elsewhere.  	Define the following basis {elements of $\g$}:
		\begin{align*}
			&e_{ab}:= \left( \begin{array}{cc} E_{ab} & 0\\
				0 & E_{ab}\\
			\end{array} \right),\hskip0.2cm f_{ab}:=	\left( \begin{array}{cc} 0&E_{ab} \\
				E_{ab} & 0\\
			\end{array} \right), \text{ for }1\leq a,b\leq n.
		\end{align*}
		Following \cite{PSqn}, we consider the even $\Z$-grading of $\mf g$ determined by $\deg(e_{ab})=\deg(f_{ab}):=2(a-b)$. The nilpotent subalgebra  $\mf m$ is then generated by $e_{ab}, f_{ab}$ for $1\leq a<b\leq n$, with character $\zeta$ of $\mf m$ determined by $\zeta(e_{a,a+1})=1$ and $\zeta(f_{a,a+1})=0$ for   $1\leq a\leq n-1$. Set  $\{u_i\}_{i=1}^{m'}:=\{e_{ab}|~1\leq a<b\leq n\},~\{u_i\}_{i=m'+1}^{m}:=\{f_{cd}|~1\leq c<d\leq n\}$ and
		\begin{align*}
			&x_{i}:= \sum_{k\geq 0}e_{b+k,a+k+1}, \text{ for }u_i=e_{ab}, \\
			&x_{i}:= \sum_{k\geq 0}(-1)^k f_{b+k,a+k+1}, \text{ for }u_i=f_{ab}.
		\end{align*}
		By a direct calculation, $\{u_i\}_{i=1}^m$ and $\{x_i\}_{i=1}^m$ satisfy Conditions (1)--(3) in Theorem \ref{thm::eqvthmpn}. 		 As a consequence, Theorem B in the case of $\mf q(n)$ follows from Theorem \ref{thm::eqvthmpn}.
\end{ex}
\begin{proof}[Proof of Theorem \ref{thm::eqvthmpn}] Using the arguments identical to that used in Lemma \ref{lem2} and Theorem \ref{thm::eqvthm}, we are left to prove that the Whittaker functor $\Whz(-):  \g\text{-Wmod}^\zeta\rightarrow   W_\zeta\text{-Mod}$ is an equivalence. The proof follows the same strategy as in the Lie algebra case given in  \cite{Skr}, with a few modifications. 
	In the following, our goal is to establish the assertions \eqref{eq::66}, \eqref{eq::67}, while omitting the parts that are analogous to the Lie algebra case and for which we refer to loc.~cit. for details. We shall adapt the arguments in the proof of \cite[Theorem 1.3]{Skr2}; see also the proof of \cite[Proposition 4.2]{WZ}.
	
	First, we define subspaces $M_{i,j}\subseteq M$ spanned by all elements $r_1r_2\cdots r_\ell v$ with $v\in \Whz(M)$ with
	\begin{align}
		&\ell\geq j,~r_1\in \g({-2+i_1}), r_2\in \g({-2+i_2}), \ldots, r_\ell \in \g({-2+i_\ell}), \label{eq::68}\\ &\text{where } i_1, i_2,\ldots, i_\ell>0 \text{ such that } i_1+i_2+\cdots i_\ell\leq i.\label{eq::69}
	\end{align} Define $M_{i,j} = M_{i,0}$ for any $i\geq 0$ and $j<0$. Also, we let $M_{i,j} = 0$ for any $i<0$.
	
	We then by definition have $M_{i,j}\subseteq M_{i',j'}$ and $\g({-2+c})M_{i,j} \subseteq M_{i+c,j+1}$, for any $i\leq i'$, $j\geq j'$ and $c>0$.
	
	 In what follows, we use $|u|$ to denote the parity of a homogenous element $u\in U(\g)$. Let $y\in \g({-d})$ with $d\geq 2$ and $y\in {\mf m}_{\oa}\cup {\mf m}_\ob$. Let  $r_1,\ldots, r_\ell\in \g$ be homogeneous elements with respect to both $\Z_2$-grading and $\Z$-grading of $\g$. If $v\in \Whz(M)$ and $r_1,\ldots, r_\ell$ satisfy Conditions \eqref{eq::68} and \eqref{eq::69}, then we have the following  super analogue of Claim 3 in the proof of \cite[Theorem 1.3]{Skr2}: there is an element $R\in M_{i-d,j}+M_{i-d-1,0}$ (depending on $y,r_1\ldots, r_\ell,v$) such that
	\begin{align}
		&(y-\zeta(y))r_1r_2\cdots r_\ell v = [y, r_1r_2\cdots r_\ell]v = \sum_{1\leq s\leq \ell,~i_s=d} c_s r_1\cdots \widehat{r_s} \cdots r_\ell [y,r_s]v+R,\label{eq::pnSk}\\
		&(y-\zeta(y))r_1r_2\cdots r_\ell v\in M_{i-d,j-1}+M_{i-d-1,0},\label{eq2::pnSk}
	\end{align}   where the notation $\widehat{r_s}$, as usual, denotes omission of $r_s$, and \[c_s = (-1)^{|y|(|r_1|+\cdots+|r_{s-1}|+|r_{s+1}|+\cdots+|r_\ell|)+|r_s|(|r_{s+1}|+\cdots +|r_\ell|)}.\] We note that $[y,r_s]\in \g({-2+i_s-d})$. Therefore, if $i_s<d$ then $[y,r_s]v=0$. Also, if $i_s=d$ then $[y,r_s]v\in \Whz(M)$. We note that the first identity in \eqref{eq::pnSk} follows from $(y-\zeta(y))r_1r_2\cdots r_\ell v = (y-(-1)^{|y|(|r_1|+\cdots+|r_\ell|)}\zeta(y))r_1r_2\cdots r_\ell v= [y, r_1r_2\cdots r_\ell]v$ since $\zeta(\mf m_\ob)=0$. Since Equations \eqref{eq::pnSk}-\eqref{eq2::pnSk} can be proved by an argument similar to that used in the proof of \cite[Thoerem 1.3]{Skr2}, we omit the proof.
	
	Suppose that $\wt{\bf b}=i$ and $|{\bf b}|=j$, then we have  $x^{\bf b}v\in M_{i,j}$. With Equations \eqref{eq::pnSk} and \eqref{eq2::pnSk}, using the same argument as in the proofs of \cite[Claims 4 and 5]{Skr2} we have $u^{\bf a}x^{\bf b}v=0$ whenever either $\wt{\bf a} >\wt{\bf b}$ or $\wt{\bf a}=\wt{\bf b}$ and $|{\bf a}|<|{\bf b}|$.
	
	Finally, we let $\wt{\bf a}=\wt{\bf b}=i$ and $|{\bf a}| =|{\bf b}|=j$. It remains to show that $u^{\bf a}x^{\bf b} v=0$ for ${\bf a}\neq {\bf b}$ and $u^{\bf a}x^{\bf a} v=cv$, for some non-zero $c\in \C$. Following the  strategy of the proof of the Claim 6 of \cite[Theorem 1.3]{Skr2}, we proceed by induction on $j$. Assume that $j>0$ and the assertions holds for smaller values of $j$. Define $p$ to be such that $a_p\neq 0$ and $a_s=0$ for any $p<s\leq m$. Denote by ${\bf e}_p$ the $m$-tuple with $1$ at the $p$-th position and $0$ elsewhere. Then we have $u^{\bf a} = u^{{\bf a}-{\bf e}_p}(u_p-\zeta(u_p))$. 
	By Equations \eqref{eq::pnSk}-\eqref{eq2::pnSk} and the facts that $u^{{\bf a}-{\bf e}_p}M_{i-d_p,j} = u^{{\bf a}-{\bf e}_p}M_{i-d_p-1,0}=0$, it follows that
	\begin{align}
		&u^{\bf a}x^{\bf b}v = \sum_{b_s>0,~d_s=d_p}q_s u^{{\bf a}-{\bf e}_p}x^{{\bf b}-{\bf e}_p}[u_p,x_s]v =\sum_{b_s>0,~d_s=d_p}q_s u^{{\bf a}-{\bf e}_p}x^{{\bf b}-{\bf e}_p}\zeta([u_p,x_s])v,
	\end{align}  for some scalars $q_s$. Since $[u_p,x_s]\in \g({-2})$ in the summation above, we have $u^{\bf a}x^{\bf b}v=0$ provided that $b_p=0$. In addition, by assumption we have $\zeta([u_p,x_s])=0$, whenever $[u_p,x_s]\in \g({-2})$ with $p\neq s$, and $[u_p,x_p]\in \g_\oa$. Therefore, if $b_p>0$ then we get \begin{align}
		&u^{\bf a}x^{\bf b}v  =c_p' b_pu^{{\bf a}-{\bf e}_p}x^{{\bf b}-{\bf e}_p}\zeta([u_p,x_p])v=c_p' b_pu^{{\bf a}-{\bf e}_p}x^{{\bf b}-{\bf e}_p}v,
	\end{align} where $c_p'=\pm 1$ depends on the parities   $|x_1^{b_1}x_2^{b_2}\cdots x_{p-1}^{b_{p-1}}|$ and $|u_p|$ as described in Equations \eqref{eq::pnSk}-\eqref{eq2::pnSk}. By induction hypothesis, Equation \eqref{eq::67} follows. If ${\bf a} ={\bf b}$, then by induction hypothesis again we get a non-zero vector $u^{{\bf a}-{\bf e}_p}x^{{\bf b}-{\bf e}_p}v \in \C v$. The conclusion in \eqref{eq::66} follows as well. This completes the proof.
\end{proof}

\subsection{Periplectic Lie superalgebra $\pn$} \label{sect::defWalgpn}
  The aim of this section is to define the principal finite $W$-superalgebra associated to the periplectic Lie superalgebra $\pn$ and establish a Skryabin type equivalence.

\subsubsection{Principal finite $W$-superalgebra for $\pn$} Recall the following matrix realization of   $\pn\subseteq \gl(n|n)$:
\[
\g=\pn=
\left\{ \left( \begin{array}{cc} A & B\\
	C & -A^t\\
\end{array} \right) \right\},
\] where $A,B,C  \in \C^{n\times n}$, $B$ is symmetric and $C$ is skew-symmetric.  For $1\le  a,b\le n$, define homogenous elements $e_{ab}, s_{ab}$ and $y_{ab}$ as follows:
\begin{align*}
&e_{ab} =  \left( \begin{array}{cc} E_{ab} & 0\\
	0 & -E_{ba}\\
\end{array} \right),~\text{  for any $1\leq a,b\leq n$};\\
&s_{ab}= \left( \begin{array}{cc} 0 & E_{ab}+E_{ba}\\
0 & 0\\
\end{array} \right),~\text{  for any $1\leq a<b\leq n$};\\
&s_{aa}= \left( \begin{array}{cc} 0 & E_{aa}\\
	0 & 0\\
\end{array} \right),~\text{  for any $1\leq a\leq n$};\\ &~y_{ab} = \left( \begin{array}{cc} 0 & 0\\
E_{ab}-E_{ba} & 0\\
\end{array} \right) ,~\text{  for any $1\leq a<b\leq n$}.
\end{align*}   Finally, we set $s_{ab}=0$ if $a,b$ do not satisfy $1\leq a\leq b \leq n$. Similarly, $y_{ab}=0$ if $a, b$ do not satisfy $1\leq a <b\leq n$. The periplectic Lie superalgebra is of type I, that is, there is a compatible $\Z$-grading $\g=\g_{1}\oplus \g_0\oplus \g_{-1}$ determined by $$\g_0=\g_\oa,~\g_1=\sum_{a\le b}\C s_{ab},~\text{and } \g_{-1}=\sum_{a<b}\C y_{ab}.$$

 Let      \begin{align*}
&f:=\sum_{a=1}^{n-1} a(n-a) e_{a,a+1},~h:= \sum_{a=1}^n(2a-1-n)e_{aa},~e:=\sum_{a=1}^{n-1} e_{a+1,a}.
\end{align*} Then $e$ is a principal nilpotent element inside the  $\mf{sl}(2)$-tuple $\langle e,f,h\rangle$. This gives rise to an even $\Z$-grading $\g=\bigoplus_{k\in 2\Z} \g(k)$ determined by
\begin{itemize}
	\item[(1)] $\deg(e_{ab}) = 2(a-b)$.
	\item[(2)] $\deg(s_{ab}) =  2(a+b)-2n-2$.
	\item[(3)] $\deg(y_{ab}) = 2n+2 -2(a+b)$.
\end{itemize} Here we use $\deg(x)$ to denote the degree of the homogenous element $x$ with respect to the $\Z$-grading of $\g$, i.e., $\deg(x)=k$ if and only if $x\in \g(k)$.
Therefore, we have
\begin{align*}
&\deg(e_{ab}) =\deg(e_{cd}) \Leftrightarrow (a,b) =(c-\ell, d-\ell), \text{ for some $\ell\in \Z$.}\\
&\deg(s_{ab}) =\deg(s_{cd}) \Leftrightarrow \deg(y_{ab}) =\deg(y_{cd}) \Leftrightarrow  (a,b) =(c-\ell, d+\ell), \text{ for some $\ell\in \Z$.}
\end{align*} We define $\mf m = \bigoplus_{ k\leq -2}\g(k)$. Define $(\cdot|\cdot)_0$ to be the trace form on $\g_\oa$, that is, $(x|y)_0:=\text{tr}(xy)$, for any $x,y\in \g_\oa$. Then the character $(\cdot|e): {\mf m_\oa}\rightarrow\C$ of $\mf m_\oa$ extends to the following character of $\mf m$:
\begin{align*}
&\zeta(e_{a,a+1})=1,\text{~for any $1\leq a\leq n-1$},\\
&\zeta(e_{a,b})=0,\text{~for $b-a>1$},
\end{align*} and $\zeta(\mf m_\ob)=0$. Define the principal finite $W$-superalgebra $  W_\zeta$ of $\g=\pn$ as in \eqref{def::generalW}.

\subsubsection{Proof of Theorem B for $\pn$}\label{sect::54}

The conclusion of Theorem B for the case of $\pn$ is a consequence of the following theorem.
	\begin{thm} \label{thm::pnmain}   Retain the notations above. There exists a homogenous basis $\{u_i\}$ for $\mf m$ and a set of homogenous elements $\{x_i\}$ in $\g$ such that Conditions (1)--(3) in Theorem \ref{thm::eqvthmpn} are satisfied.
	 In particular, $\mc N(\eta)$ and $W_\zeta\emph{-fdmod}$ are equivalent, for any non-singular $\eta\in \ch\mf m_\oa$.
	\end{thm}

Before giving the proof of Theorem  \ref{thm::pnmain}, we need the following useful lemmas.
\begin{lem} \label{lem::30}
	For any $1\leq a\leq b\leq n$ and $1\leq c< d\leq n$ such that $[s_{ab},y_{cd}]\in \mf m$, we have
	\begin{align*}
		&\zeta([s_{ab},y_{cd}])=\left\{\begin{array}{lll}1, &  \text{~for } (c,d) =(a,b+1);\\
			-1, & \text{for } (c,d) =(a+1,b);\\
			0, & \text{~otherwise}.
		\end{array} \right.
	\end{align*}
\end{lem}
\begin{proof}
 Since this can proved by a direct calculation, we omit the proof.
\end{proof}

For any $1\leq a\leq b\leq n$ and $1\leq c<d\leq n$, we define
\begin{align}
	&\ov s_{ab}:=\sum_{k\geq 0} y_{a-k,b+k+1},~\ov y_{cd}:=\sum_{k\geq 0} s_{c+k,d-k-1}.
	\end{align}
Note that $[s_{ab},\ov s_{ab}],~[y_{cd},\ov y_{cd}]\in \g(-2)$.
\begin{lem} \label{lem::34} We have
	\begin{align}
	&\zeta([s_{ab},\ov s_{ab}])=1, \label{eq::617} \\
	&\zeta([s_{a+\ell,b-\ell},\ov s_{ab}])=0, \text{ for any }\ell\neq 0,\label{eq::618} \\
	&\zeta([y_{cd},\ov y_{cd}])=1, \label{eq::619}\\
	&\zeta([y_{c+\ell,d-\ell},\ov y_{cd}])=0, \text{ for any }\ell\neq 0. \label{eq::620}
	\end{align}
\end{lem}
\begin{proof}
By Lemma \ref{lem::30}, we have $\zeta([s_{ab},\ov s_{ab}]) = \zeta([s_{ab}, y_{a,b+1}])=1$ and $\zeta([s_{a+\ell,b-\ell},\ov s_{ab}])=0$, for any $\ell>0$. In the case that $\ell<0$, by Lemma \ref{lem::30} we get
\begin{align*}
&\zeta([s_{a+\ell,b-\ell}, \ov s_{ab}]) = \zeta([s_{a+\ell,b-\ell},  y_{a+\ell, b-\ell+1}])+\zeta([s_{a+\ell,b-\ell}, y_{a+\ell+1, b-\ell}])=0.
\end{align*} The equations \eqref{eq::619}, \eqref{eq::620} can be proved similarly. This completes the proof.
\end{proof}

\begin{proof}[Proof of Theorem \ref{thm::pnmain}]
	Since $\mc N(\eta)\cong \mc N(\zeta)$ by  Proposition \ref{prop::equivNz}, we may assume that $\eta=\zeta$. 
	 	 Set $m:=\dim \mf m$ and define
	\begin{align*}
	&\{u_i\}_{i=1}^m:=\{e_{ab}|~\deg(e_{ab})<0\}\cup \{s_{ab}|~\deg(s_{ab})<0\}\cup \{y_{cd}|~\deg(y_{cd})<0\}.
	\end{align*} Then $\{u_i\}_{i=1}^m$ forms a homogenous basis for $\mf m$.

For any $1\leq a< b\leq n$, define $$\ov e_{ab}:=\sum_{k\geq 0}e_{b+k,a+k+1}$$ so that  $[e_{ab}, \ov e_{ab}]\in \g(-2)$ and $\zeta([e_{ab},\ov e_{ab}])=1$ and $\zeta([e_{ab},\ov e_{cd}])=0$, whenever $[e_{ab},\ov e_{cd}]\in \g(-2)$ with $e_{ab},e_{cd}$ different. We shall check that Conditions (1)--(3) in Theorem \ref{thm::eqvthmpn} are satisfied for $x_i:=\ov u_i$, for $1\leq i\leq m$. By Lemma \ref{lem::34}, it remains to show that $\zeta([u_i, \ov u_j])=0$ whenever $[u_i,\ov u_j]\in \g(-2)$ with $i\neq j$.

 If $[u_i, \ov u_j]\in \g_\ob$ then the conclusion follows  since $\zeta(\mf m_\ob)=0$. Therefore, it remains to consider the case $u_i, u_j\in \g_\ob$.  If $u_i\in \g_k,  u_j\in \g_{-k}$, for some $k=\pm 1$, then $\ov u_j\in \g_{k}$ and therefore $[u_i,\ov u_j]=0$. If $u_i = s_{ab}$ and $u_j =s_{cd}$, then  by assumption we get
\begin{align*}
&\deg(s_{ab})+\deg(\ov s_{cd}) = -2 =\deg(s_{ab})+\deg(\ov s_{ab}).
\end{align*} This implies $\deg(\ov s_{cd})= \deg(\ov s_{ab})$ and so $(a,b) =(c+\ell,d-\ell)$, for some integer $\ell\neq 0$.  Consequently, $\zeta([s_{ab}, \ov s_{cd}])=0$ by \eqref{eq::618} in Lemma \ref{lem::34}. The  conclusion for the case that  $u_i = y_{ab}$ and $u_j =y_{cd}$ can be proved similarly.   The conclusion follows from Theorem \ref{thm::eqvthmpn}.
\end{proof}

	Recall the standard Whittaker modules $$M(\la,\zeta):=K(M_0(\la,\zeta))=\Ind_{\g_0+\g_1}^\g(M_0(\la,\zeta))$$ from \cite[Section 3.2]{Ch21}. We have the following classification of simple $W_\zeta$-modules:
\begin{cor}
	The set
	\begin{align*}
	&\{\emph{Wh}_\zeta(M(\la,\zeta))|~\la\in \h^\ast \text{ is anti-dominant}\}
	\end{align*} is a complete set of mutually non-isomorphic simple modules of $ W_\zeta$. In particular, any simple $W_\zeta$-module is finite-dimensional.
	\end{cor}

 \begin{ex}
 	We explain the example of simple $W_\zeta$-modules for $\g =\mf p(2)$. In this case we have $\mf m= \C e_{12}+\C s_{11}$ and $\zeta(e_{12})=1,~\zeta(s_{11})=0$. We regard $M_0(\la,\zeta)$ as a $\g_\oa$-submodule of $M(\la,\zeta)$. Let $v_\la\in \text{Wh}^0_\zeta(M_0(\la,\zeta))$ be a non-zero Whittaker vector. Since $\Res M(\la,\zeta)\cong \Lambda(\g_{-1})\otimes M_0(\la,\zeta)$, we have $\text{Wh}_\zeta^0(\Res M(\la,\zeta)) = \C v_\la+\C y_{12}v_\la$ by \cite[Theorem 4.6]{Ko78}; see also \cite[Example 32]{Ch21}. We calculate
 	\[s_{11}y_{12}v_\la= e_{12}v_\la \neq 0. \]
 	As a consequence, any simple $W_\zeta$-module is of the form  $\text{Wh}_\zeta(M(\la,\zeta))=\C v_\la$.  	We remark that this subspace coincides with $\text{Wh}_\zeta^0(M(\la,\zeta))\cap \{m\in M(\la,\zeta)|~\g_1m=0\}$ by \cite[Example 32]{Ch21}, where a different definition of Whittaker vectors was considered.
 \end{ex}

	\section{Super Soergel Struktursatz for functor $\Whz\circ\Gamma_\zeta(-)$}\label{sect::6}
	
	In this section, we let $\mf g$ be a basic classical Lie superalgebra or a strange Lie superalgebra. Let $W_\zeta$ denote the principal finite $W$-superalgebra considered in Sections \ref{sect::FWalg} and \ref{sect::pn}. For a given subset $\ups\subseteq {\h^\ast_\oa}$, we let  $\fdWmod_{\ups}$ be the Serre subcategory of $\fdWmod$ generated by composition factors of  modules $\Whz(\Gamma_\zeta(L(\la)))$, for $\la \in \ups$.
	
		\subsection{Block decomposition of $\fdWmod$} \label{sect::61} In this subsection, we consider basic classical Lie superalgebra $\g$.
		
		\subsubsection{Central blocks} Recall that $\pr: U(\g)\rightarrow U(\g)/I_\zeta$ denotes the natural projection. We note that $\pr(Z(\g))$ is a subalgebra contained in the center of $W_\zeta$.  By Theorem \ref{thm::eqvthm}, every simple $W_\zeta$-module admits a character of $\pr(Z(\g))$. Therefore, the category $\fdWmod$ admits a central block decomposition $$\fdWmod = \bigoplus_{\chi} \fdWmod_\chi,$$ according to characters $\chi$ of $Z(\g)$. Here   $\fdWmod_\chi$ is the full subcategory of {$W_\zeta\text{-fdmod}$ of} $W_\zeta$-modules on which $\pr(x)-\chi(x)$ acts locally nilpotently, for any $x\in Z(\g)$.  By Theorem \ref{thm::eqvthm}, $\fdWmod_\chi$ is the Serre subcategory of $\fdWmod_\chi$ generated by composition factors of $\Whz(\Gamma_\zeta(L(\la)))$ with $\chi_\la =\chi$.

We are going to give a finer decomposition.  Let $\chi$ be a central character of $\g$, and define $\fdWmod_{\ups,\chi}:=\fdWmod_{\ups}\cap \fdWmod_{\chi}$. Then, we have $$\fdWmod_\chi = \bigoplus_{\ups\in \h^\ast/\sim} \fdWmod_{\ups,\chi},$$  where $ \fdWmod_{\ups,\chi}\cong \mc N(\ups,\zeta)_\chi $ and the equivalence relation $\sim$ on $\h^\ast$ is given in \eqref{eq::sim}. The following is a consequence of Theorems A and B in Section \ref{sect::intro}.

\begin{cor} Let $\ups\subseteq \Lambda$ be closed under $\sim$. Then the following holds.
	\begin{itemize}
		\item[(1)]  The set  $\{\emph{Wh}_\zeta(\Gamma_\zeta(L(\la))|~\la\in \ups\text{ is $\Pi_\zeta$-free} ,~\chi_\la=\chi\}$ is an exhaustive list of non-isomorphic simple modules in $\emph{W}_\zeta\emph{-fdmod}_{\ups,\chi}$.
		\item[(2)] Suppose that  $\la \in \h^\ast$ is strongly typical.
		Then $\emph{Wh}_\zeta(\Gamma_\zeta(L(\la)))$ is the unique simple module in $W_\zeta\emph{-fdmod}_{\chi_\la}$. In this case, $W_\zeta\emph{-fdmod}_{\chi_\la}$ is equivalent to an indecomposable block of the category of finite-dimensional modules over $Z(\g_\oa)$.
	\end{itemize}
\end{cor}

\begin{rem}
	Restricting the generalized Gelfand-Graev module  $Q_\zeta$ to a   $(Z(\g_\oa), W_\zeta)$-bimodule the functor   $Q_\zeta\otimes_{  W_\zeta}-$ defines an exact functor from $\fdWmod$ to the category locally finite $Z(\g_\oa)$-modules. For an given   simple $W_\zeta$-module $S$, the following are equivalent
	\begin{itemize}
		\item[(1)] $S\cong \Whz(\Gamma_\zeta(L(\la)))$, for some $\la\in \Lambda$.
		\item[(2)] $Q_\zeta\otimes_{  W_\zeta} S$ contains a one-dimensional $Z(\g_\oa)$-module induced by the a central character of $\g_\oa$ associated to an integral weight.
	\end{itemize}
\end{rem}

{\subsubsection{Indecomposable blocks  for $\gl(m|n)$} \label{Sect::612}
	In this subsection, we let $\g=\gl(m|n)$.
 We give a  combinatorial description   of simple objects in an (arbitrary) indecomposable block of $\fdWmod$ as follows.
To explain this in more detail, recall the elementary matrices $E_{ij}$ in the matrix realization of $\g$ from \eqref{gllrealization}.    Let $\mf h$ be the Cartan subalgebra spanned by $E_{ii}$, for $1\leq i\leq m+n$ with dual basis $\vare_i\in \h^\ast$ determined by $\vare_i(E_{jj})=\delta_{ij}$, for $1\leq i,j\leq m+n$. Set  $\langle\cdot, \cdot\rangle:\h^\ast\times \h^\ast\rightarrow \C$  to be the non-degenerate  bilinear form induced by the super-trace $\str$, namely,  $\vare_i$ are mutually orthogonal for $1\leq i\leq m+n$ and $\langle\vare_j,\vare_j\rangle=1, \langle\vare_k,\vare_k\rangle=-1$, for $1\leq j\leq m$ and $m+1\leq k\leq m+n$. We set $\mc O$ to be the BGG category with respect to the Borel subalgebra $\mf b$ spanned by $E_{ij}$, for $1\leq i\leq j\leq m+n$.

  Define an equivalence relation $\approx$ on $\h^\ast$ by declaring
 \begin{align*}
  &\la\approx\mu, \text{ for }\la,\mu\in\h^\ast,
 \end{align*}
if there exist mutually orthogonal odd roots $\alpha_1,\ldots,\alpha_\ell$, integers $c_1,\ldots,c_\ell$ and an element $w\in W$ such that
\begin{align*}
&\mu+\rho=w(\la+\rho-\sum_{i=1}^\ell c_i\alpha_i),~ \langle \la+\rho,\alpha_i\rangle=0,~1\leq i\leq \ell.
\end{align*}
Recalling relation $\sim$ from \eqref{eq::sim} we note that $\approx$ is finer than $\sim$, i.e., $\la\approx\mu$ implies that $\la\sim\mu$, since $W(\mf l_\zeta)=W$. Recall the simple Whittaker modules $L(\la,\zeta)\cong \Gamma_\zeta(L(\la))$, for $\la\in \h^\ast$, from Section \ref{Sect::31}. The following proposition gives a  description of the linkage principle of $\fdWmod$ in terms of $\approx$.
 \begin{prop} \label{Thm::29} The subcategories $\mc N(\ups,\zeta)_{\chi}$ and $W_\zeta\emph{-fdmod}_{\ups,\chi}$ are  indecomposable for every $\ups \in \h^\ast/\sim$. Furthermore, if $\la\in \ups$ such that $\chi_\la =\chi$, then the following set
 	\begin{align*}
 	&\{\emph{Wh}_\zeta(L(\mu,\zeta))|~\mu\approx\la\},
 	\end{align*}  is a complete set of simple objects in $W_\zeta\emph{-fdmod}_{\ups,\chi}$.
 \end{prop}
\begin{proof}
	Let $\la, \mu\in \h^\ast$. We shall show that the following two conditions are equivalent:
	\begin{itemize}
		\item[(1)] $\la\approx\mu$.
		\item[(2)] $\la\sim\mu$ and $\chi_\la=\chi_\mu$.
	\end{itemize}
 To prove the equivalence of $(1)$ and $(2)$, we first recall the description of central characters from \cite[Section 2.2.6]{ChWa12}, which is a consequence of the description of $Z(\g)$; see also \cite{Kac84, Se}. Namely,  $\chi_\la=\chi_\mu$ if and only if  there exist mutually orthogonal odd roots $\alpha_1,\ldots,\alpha_\ell$, complex numbers  $b_1,\ldots,b_\ell$ and $w\in W$ such that  $\mu+\rho=w(\la+\rho-\sum_{i=1}^\ell b_i\alpha_i)$,  and $\langle \la+\rho,\alpha_i\rangle=0$ for $1\leq i\leq \ell.$ This proves the implication $(1)\Rightarrow (2)$.

 Conversely, assume that $\la\sim\mu$ and $\chi_\la=\chi_\mu$. Namely, we have  $\la-w\cdot \mu \in \Z\Phi$ and $\chi_\la =\chi_{\mu}=\chi_{w\cdot \mu}$, for some $w\in W$. By the same argument as given in the proof of  \cite[Proposition 3.3]{CMW}, it follows that $\la\approx w\cdot \mu\approx \mu$. This establishes the equivalence of (1) and (2).
	
	By Theorem \ref{thm::eqvthm}, it suffices to show that $L(\la,\zeta)$ and $L(\mu,\zeta)$ lie in the same indecomposable block of $\mc N(\zeta)$ if and only if
	$\la\approx \mu$. If $L(\la,\zeta)$ and $L(\mu,\zeta)$ are in the same indecomposable block, then $\chi_\la=\chi_\mu$ and $\la\sim\mu$ by Proposition \ref{prop::3}, which implies that $\la\approx\mu$.
	
	Suppose that $\la\approx\mu$ such that $\mu = w\cdot (\la-\sum_{i=1}^\ell c_i\alpha_i)$, for some $w\in W$ and some integers $c_1,\ldots,c_\ell$ and  mutually orthogonal odd roots $\alpha_1,\ldots,\alpha_\ell$ such that $\langle \la+\rho, \alpha_i\rangle =0,$ for $1\leq i\leq \ell$. Without loss of generality, we may assume that $\ell=1$ and set $\alpha:=\alpha_1$ now, since the general case can be proved by induction on $\ell$. In addition, we may assume that $w$ is the identity of $W$, since $L(\mu,\zeta)\cong L(\la-\sum_{i=1}^\ell c_i\alpha_i,\zeta)$. If $\alpha$ is a simple root, i.e., $\alpha=\vare_{m}-\vare_{m+1}$, then the Verma module $M(\la)$ has $L(\la-\alpha)$ as a composition factor; see, e.g., the proof of  \cite[Theorem 3.12]{CMW}. This implies that the standard Whittaker module $M(\la,\zeta)$ has $L(\la-\alpha,\zeta)$ as a composition factor. Since $M(\la,\zeta)$ is indecomposable, it follows that $L(\la,\zeta)$ and $L(\la-\alpha,\zeta)$ lie in the same indecomposable block. If $-\alpha$ is simple, we reverse the role of $\la$ and $\la-\alpha$ and reach the same conclusion. Finally, suppose that $\alpha\neq \vare_m-\vare_{m+1}$. We pick $x\in W$ such that $x\alpha = \vare_m-\vare_{m+1}$. We may observe that $$L(\la-\alpha,\zeta)\cong L(x\cdot(\la-\alpha),\zeta)=L(x\cdot\la-x\alpha,\zeta).$$ Since $L(x\cdot\la-x\alpha,\zeta)$ and $L(x\cdot\la,\zeta)\cong L(\la,\zeta)$ lie in the same indecomposable block, the conclusion follows. This completes the proof.
\end{proof}

We remark that, in the case $\ups \subseteq \Lambda$, Proposition \ref{Thm::29} also follows from Theorem \ref{thm::eqvthm} and \cite[Theorem C]{Ch212}.
}

	\subsection{The category $\fdWmod^n_\Lambda$}
	\label{sect::62}
	
	In this subsection, we let $\g$ be a basic classical Lie superalgebra or a strange Lie superalgebra. Recall from Section \ref{sect::112} that  $\mf p=\mf l_\zeta+\mf n_\oa$ is a parabolic subalgebra of $\g_\oa$ containing $\mf l_\zeta$ as a Levi subalgebra. For a positive integer $n$ and a weight $\la \in \h^\ast_\oa$, we define $M^n_0(\la,\zeta):=U(\g_\oa)\otimes_{U(\mf p)}Y^n_\zeta(\la,\zeta)\in \mc N_0(\zeta)$, where  $Y_\zeta^n(\la, \zeta):=U(\mf l_\zeta)/(\text{Ker}\chi_\la^{\mf l_\zeta})^n U(\mf l_\zeta)\otimes_{U({\mf n}\cap \mf l_\zeta)}\mathbb C_\zeta$; see  \cite[Section 5]{MS}. We note that $M^1_0(\la,\zeta)=M_0(\la,\zeta)$.
	
	 We consider the full subcategory $\fdWmod_\Lambda$ of $\fdWmod$.   Let $\nu\in \Lambda$ be fixed under the dot-action of $W$. For any positive integer $n$, we define $\fdWmod^n_\Lambda$ to be the full subcategory of subquotients of $\Whz(E\otimes \Ind {M^n_0}(\nu,\zeta))$, for any finite-dimensional $\g$-module $E$.

	Let $\Theta: \g\mod \rightarrow \g\mod$ be a projective functor,  then we can define the corresponding projective functor $\Theta^\zeta:= \Whz\circ  \Theta\circ Q_\zeta\otimes_{  W_\zeta} - : \fdWmod\rightarrow \fdWmod$.  These functors have been considered in literature in the case when $\g$ is reductive; see,  e.g., \cite[Section 8]{BK08}, \cite{Goodwin11}. We collect some basic properties of  $\fdWmod^n_\Lambda$ in the following lemma.
	\begin{lem}
		For any $n\geq 1$ we have \begin{itemize}
			\item[(1)] {\em $\fdWmod_\Lambda = \bigcup_{k\geq 1} \fdWmod^k_\Lambda$}.
			\item[(2)] {\em  $\fdWmod^n_\Lambda$} has enough projective modules.
				\item[(3)]  {\em  $\fdWmod^n_\Lambda$} is stable under applying projective functors.
			\item[(4)]  The $W_\zeta$-module $\Theta^\zeta(\emph{Wh}_\zeta(\Ind {M^n_0}(\nu,\zeta)))$ is projective in {\em  $\fdWmod^n_\Lambda$}. Furthermore,  all direct summands of $\Theta^\zeta(\emph{Wh}_\zeta(\Ind {M^n_0}(\nu,\zeta)))$, for projective functors $\Theta:\g\mod\rightarrow \g\mod$,  constitute all projective modules in {\em $\fdWmod^n_\Lambda$}.
		\end{itemize}
	\end{lem}
	\begin{proof} Using Theorem \ref{thm::eqvthm}, the  conclusions follow from  \cite[Theorem 16]{Ch21} and \cite[Lemma 14]{Ch21}.
	\end{proof}
	
Now, we put the results in previous sections together to give a proof of Theorem C as follows.
 \begin{proof}[Proof of Theorem C]
 	By Theorems \ref{thm::4} and \ref{app::thmqn}, $\Gamma_\zeta(-),~F_\nu(-):\mc O_\Lambda \rightarrow \mc N(\zeta)$ are isomorphic and satisfy the universal property of the Serre quotient functor $\mc O_\Lambda\rightarrow {\mc O_{\Lambda}}/\mc I_\nu$. By Proposition \ref{prop::proF}, it follows that $\Gamma_\zeta(\mc O_\Lambda) = \mc W_\nu(\zeta)$ and so
 	$\fdWmod^1_\Lambda=\Whz(\Gamma_\zeta(\mc O_\Lambda))$. The first conclusion follows by Theorem B.
 	   Finally, recall the super version of Soergel functor $\mathbb V^{sup}$ as defined in \cite[Section 7.4.3]{CCM2} and \cite[Section 9]{CCM2}. Since the functor $\Whz\circ \Gamma_\zeta(-)$ is, up to an equivalence, isomorphic to  $\mathbb V^{sup}$ by \cite[Corollary 54]{CCM2}, we conclude from  \cite[Theorem 7.2]{AM} that it is fully faithful on projective modules.
 \end{proof}

\appendix

\section{Whittaker categories of queer Lie superalgebras} \label{sect::appA}
		The goal of this section is to prove an analogue of Theorem A Part (2) from Section \ref{sect::intro} for the queer Lie superalgebra $\mf g:=\mf q(n)$.
		
		Recall the generators $e_{ij}$ and $f_{ij}$ for $1\leq i,j\leq n$ from Example \ref{eg::25}.  Set $\mf b$ to be the Borel subalgebra  spanned by   $\{e_{ij}, f_{ij}|~1\leq i \leq  j\leq n\}$ with the Cartan subalgebra $\h$ spanned by $e_{ii}, f_{ii}$ for $1\leq i\leq n$, respectively.  Let $\{\vare_i| 1\leq i\leq n\}$ be the dual basis of $\mathfrak{h}_{\bar{0}}^{*}$ determined by $\vare_j({e_{ii}})=\delta_{ij}$. We let $\mc O$ denote the BGG category with respect to $\mf b$.
		
		 A weight $\la  = \sum_{i=1}^n\la_i\vare_i\in \h^\ast_\oa$ is called regular dominant strongly typical in the sense of \cite{FM09}  if the following conditions are satisfied:
		 \begin{itemize}
		 	\item[(1)]$\la - w\la$ is a non-zero sum of positive roots, for any element $w$ in the integral Weyl group of $\la$.
		 	\item[(2)] $\la_i+\la_j\neq 0$ and $\la_i\neq 0$, for any $1\leq i\neq j\leq n$.
		 \end{itemize}   It is proved by Frisk and Mazorchuk in \cite{FM09} that, for such a weight $\la$, there exists  a positive integer $k$ such that $\Ind(-)_{\chi}$ and $\Res(-)_{\chi^0}$ decompose into a direct sum of $k$ copies of some functors $F_1$ and $G_1$, respectively,   as functors between  $\mc O_{\chi_\la}$ and $\mc O^\oa_{\chi_\la^0}$. Furthermore,  $F_1$ and $G_1$ give rise to mutually inverse equivalences such that
	 $\Ind M_0(\la)_{\chi_{\la}} = M(\la)^{\oplus k};$ see also  \cite[Theorem 1]{FM09} and \cite[Proposition 2]{FM09}.
	
	 	Let $\mf n$ be the nilpotent radical of $\mf b$. Fix a character $\zeta$ of $\mf n_\oa$. Let  $\nu \in \h^\ast_\oa$ be a dominant weight such that under the dot-action of $W$ its stabilizer subgroup is $W(\mf l_\zeta)$.  Then there exists a generic weight $\la \in \nu+\Lambda$ such that $\Ind M_0(\la)_{\chi_{\la}} = M(\la)^{\oplus k},$ for some positive integer $k$. We have the following consequence.

		\begin{thm} \label{app::thmqn}  We have isomorphic functors $$\Gamma_\zeta(-)\cong F_\nu(-):\mc O_{\nu+\Lambda}\rightarrow \mc N(\zeta),$$ and they satisfy the universal property of quotient functor from $\mc O_{\nu+\Lambda}$ to the quotient category by the Serre subcategory ${\mc I}_{\nu}$, up to an equivalence  between   $\mc W_\nu(\zeta)$  and $\mc O_{\nu+\Lambda}/\mc I_\nu$. Consequently, the set $			\{\Gamma_\zeta(L(\mu))|~\mu\in \Lambda(\nu)\}$	
				is an exhaustive list of mutually non-isomorphic simple Whittaker modules in $\mc N(\zeta)_{\nu +\Lambda}$.
		\end{thm}\begin{proof}  The theorem can be proved following a similar strategy as the one used in the proof of Theorem \ref{thm::4}. Our goal is to prove that the two functors $\Gamma_\zeta(-)$ and $F_\nu(-)$ from $\mc O_{\nu+\Lambda}$ to $\mc N(\zeta)$ are isomorphic.  As mentioned above,  there exists a weight $\la \in \nu+\Lambda$ such that $\Ind M_0(\la)_{\chi_{\la}} = M(\la)^{\oplus k},$ for some positive integer $k$.  We calculate that
		\begin{align*}
			\Gamma_\zeta(M(\la)^{\oplus k})\cong \Gamma_\zeta(\Ind M_0(\la))_{\chi_\la}\cong \Ind M_0(\la,\zeta)_{\chi_\la}\cong  F_\nu(\Ind M_0(\la))_{\chi_\la} \cong F_\nu(M(\la)^{\oplus k}).
		\end{align*}
		By the Krull--Schmidt theorem, we have $	\Gamma_\zeta(M(\la))\cong  F_\nu(M(\la)).$  The conclusion now follows by an analogous argument as in Theorem \ref{thm::4} using Lemma \ref{lem::fpfO}.
	\end{proof}
The study of annihilator ideas of simple Whittaker modules goes back to \cite[Theorem 3.9]{Ko78}. The following corollary is analogue of \cite[Corollary 35]{CC22} for $\mf q(n)$.
\begin{cor}
	Let $\mu\in \Lambda(\nu)$. Then the annihilator ideals of $\Gamma_\zeta(L(\mu))$ and $L(\mu)$ are the same.
\end{cor}

\section{Equivalence of the categories $\mc N(\zeta)$} \label{sect::app}
Let $\g$ be a basic classical or a strange Lie superalgebra. The goal of this section is to establish several equivalences of the categories $\mc N(\zeta)$ with respect to different Borel subalgebras and characters $\zeta$. To make a distinction, for a given Borel subalgebra $\mf b = \mf h+\mf n$, we let $\mc N^{\mf b}$ denote the category of Whittaker modules for $\g$ with respect to $\mf b$. Similarly, for a character $\zeta\in \mf n_\oa$, we define $\mc N^{\mf b}(\zeta)$ as the full subcategory of $\mc N^{\mf b}$ of modules on which $x-\zeta(x)$ acts locally nilpotently, for any $x\in \mf n_\oa$.

The following proposition is the main result in this section.

 \begin{prop} \label{prop::equivNz}
 	Let $\mf b$ and $\mf b'$ be two Borel subalgebras of $\g$. Then the Whittaker categories $\mc N^{\mf b}$ and $\mc N^{\mf b'}$  are equivalent. Furthermore,  if $\mf n$ is the nilradical of $\mf b$ and $\zeta, \eta\in \ch \mf n_\oa$ such that $\mf l_\zeta=\mf l_\eta$. Then, $\mc N^{\mf b}(\zeta)$ and $\mc N^{\mf b}(\eta)$ are equivalent.  \end{prop}
Before giving a proof of Theorem \ref{thm::4}, we need the following two preparatory lemmas for the cases that $\g=\pn, \mf q(n)$. Recall the generators $e_{ij}, s_{ij}, y_{ij}$ of $\pn$ from Section \ref{sect::defWalgpn}, and recall   the generators $e_{ij}$ and $f_{ij}$ of $\mf q(n)$ from Example \ref{eg::25}.
\begin{lem}\label{lem::autopn}
	Suppose that $\g=\pn$. Let ${\bf a}:=(a_1,\ldots,a_{n-1})$ be a sequence  of non-zero complex numbers. We define the  following  complex numbers
	\begin{align}
		&a_{ij}:= a_{i}a_{i+1}\cdots a_{j-1}, \quad a_{ji}:=a_{ij}^{-1}=a_{j-1}^{-1}a_{j-2}^{-1}\cdots a_{i}^{-1}, \quad {a_{kk}}:=1, \label{eq::aseq}
	\end{align} for any $1\leq i<j\leq n$ and {$1\leq k\leq n$}. 	Let  $\phi:=\phi_{\bf a}:\g\rightarrow \g$ be a linear map determined by
\begin{align}
&\phi:~e_{ij}\mapsto a_{ij}e_{ij},\\
&~s_{ii} \mapsto s_{ii},~s_{pq}\mapsto a_{pn}a_{qn}s_{pq},\\
&y_{pq} \mapsto a_{pn}^{-1}a_{qn}^{-1}y_{pq},
\end{align}  for any $1\leq i,j,k\leq n$ and $1\leq p<q\leq n$. Then $\phi$ is an automorphism of $\g$.
\end{lem}
\begin{proof}
	By a direct computation, we have
	\begin{align}
		&a_{ij}a_{jk} = a_{ik}, \label{eq::116}
	\end{align} for any $1\leq i,j,k\leq n$. This implies that $\phi$ restricts to an automorphism on $\g_\oa =\gl(n)$; {(see also \cite[Lemma 5.5.9]{Mu12})}. 

Let  $1\leq i,j,k\leq n$ and $1\leq p<q\leq n$. We shall deduce that $\phi([x,y])=[\phi(x),\phi(y)]$, for any $x,y = e_{ij}, s_{ii}, s_{pq}, y_{pq}$.
{Using \eqref{eq::116}}, we calculate
\begin{align*}
&[\phi(e_{ij}),\phi(s_{pq})] = a_{ij}a_{pn}a_{qn}[e_{ij},s_{pq}]=\left\{\begin{array}{lll}a_{in}a_{qn} s_{iq}, &  \text{~for } j=p\text{ and }i\leq q;\\
	a_{in}a_{qn} s_{qi}, &  \text{~for } j=p\text{ and }q\leq i;\\
	a_{in}a_{pn} s_{ip}, & \text{~for } j=q \text{ and }i\leq p;\\
		a_{in}a_{pn} s_{pi}, & \text{~for } j=q \text{ and }p\leq i;\\
	0, & \text{~otherwise}.
\end{array} \right.
\end{align*} This is equal to $\phi([e_{ij},s_{pq}])$. Also, we have  $[\phi(e_{ij}),\phi(s_{kk})] =\phi([e_{ij},s_{kk}])$.

Next, using \eqref{eq::116} it follows that
\begin{align*}
&[\phi(e_{ij}),\phi(y_{pq})] = a_{ij}a_{pn}^{-1}a_{qn}^{-1}[e_{ij},y_{pq}]=\left\{\begin{array}{lll}-a_{jn}^{-1}a_{qn}^{-1} y_{jq}, &  \text{~for } i=p\text{ and }j<q;\\a_{jn}^{-1}a_{qn}^{-1} y_{qj}, &  \text{~for } i=p\text{ and }j>q;\\
	 a_{jn}^{-1}a_{pn}^{-1} y_{jp}, & \text{~for } i=q\text{ and }j<p;\\
	 -a_{jn}^{-1}a_{pn}^{-1} y_{jp}, & \text{~for } i=q\text{ and }j>p;\\
	0, & \text{~otherwise}.
\end{array} \right.
\end{align*} This is equal to $\phi([e_{ij},y_{pq}])$.

Similarly, we may calculate that
\begin{align*}
&[\phi(s_{ij}), \phi(y_{pq})] = a_{in}a_{jn}a_{pn}^{-1}a_{qn}^{-1}[s_{ij}, y_{pq}] =  a_{in}a_{jn}a_{np}a_{nq}[s_{ij}, y_{pq}] = \phi([s_{ij}, y_{pq}]).
\end{align*}
This completes the proof.
\end{proof}

The following lemma can be verified directly as Lemma \ref{lem::autopn}.
\begin{lem} \label{lem::qnauto} Let $\mf g=\mf q(n)$.  Let ${\bf a}:=(a_1,\ldots,a_{n-1})$ be a sequence  of non-zero complex numbers.  Define the  complex numbers $a_{ij}$ for any $1\leq i , j\leq n$ as in \eqref{eq::aseq}. 		Let $\phi:=\phi_{\bf a}:\g\rightarrow \g$ be the linear map determined by
		\begin{align}
			&\phi:~e_{ij}\mapsto a_{ij}e_{ij},~f_{ij}\mapsto a_{ij}f_{ij},
		\end{align}  for any $1\leq i,j\leq n$. Then $\phi$ is an automorphism of $\g$.
\end{lem}

 \begin{proof}[Proof of Proposition \ref{prop::equivNz}]
 First, by extension of the action of $W$, the Borel subalgebra $\mf b$ is conjugate to a Borel subalgebra that has  $\mf b_\oa'$ as underlying even subalgebra, see, e.g., \cite[Chaper 3]{Mu12} or \cite[Section 1.3]{CCC}. It follows that $\mc N^{\mf b}\cong \mc N^{\mf b'}$.

 Next, assume that $\g$ is basic classical. Then $\g$ is contragredient and generated by Chevalley generators $h_i, e_{i}, f_i$, for some $1\leq i\leq \ell$. Fix a sequence of non-zero complex numbers ${\bf a}:=(a_1, a_2, \ldots, a_\ell)$. By \cite[Section 2.5]{Kac77} the following defines an automorphism $\phi_{\bf a}$ of $\g$:
 \begin{align*}
 	&\phi_{\bf a}:~h_i\mapsto h_i,~e_{i}\mapsto a_{i}e_{i},~f_{i}\mapsto a_{i}^{-1}f_{i},
 \end{align*}  for any $1\leq i\leq \ell$.  By the argument above, without of loss of generality we may assume that $\mf n$ is generated by $e_i$, for $1\leq i\leq \ell$.

 Define a  character $\hat \zeta\in \ch \mf n_\oa$ by letting
$\hat \zeta(e_{i}) \neq 0 \Leftrightarrow \hat \zeta (e_{i}) =1 \Leftrightarrow \zeta(e_{i})\neq 0.$
 It suffices to show that $\mc N^{\mf b}(\zeta)\cong \mc N^{\mf b}(\hat \zeta)$. Define the following non-zero complex numbers:
\begin{align*}
	&a_i := \left\{ \begin{array}{ll} \zeta(e_{i})^{-1},\quad \text{if $\zeta(e_{i})\neq 0$,} \\
		1, \quad\text{if $\zeta(e_{i}) =0$}.\end{array} \right.
\end{align*}
Let ${\bf a}:=(a_1,a_2,\ldots,a_\ell)$. Denote by  $\phi_{\bf a}$ the automorphism  of $\g$ as described above. This induces an auto-equivalence $T:=T_\zeta$ on $\g\Mod$. We shall show that $T:\mc N^{\mf b}(\zeta)\xrightarrow{\cong}\mc N^{\mf b}(\hat \zeta)$ is an equivalence. {Let $M\in \g\Mod$, then} $T(M)$ has the same underlying subspace as $M$, and if we  use the star-notation $\ast$ to denote the $\g$-action on $T(M)$, then the $\g$-module structure is given as follows:
\begin{align*}
	&x\ast m:= \phi_{\bf a}(x)m,
\end{align*} for any $x\in \g$ and $m\in T(M)$.
We note that $T(M)$ is finitely-generated over $\g$ if and only if $M$ is, too. Next, $T(M)$ is locally finite over $Z(\g_\oa)$ if and only if $M$ is, too, since $\phi_{\bf a}$ restricts to an automorphism on $Z(\g_\oa)$. Finally, we shall prove that $T(\mc N^{\mf b}(\zeta)) = \mc N^{\mf b}(\hat \zeta)$. Let $M\in \mc N^{\mf b}(\zeta)$. Then $M$ is generated by a   set of Whittaker vectors $v_1,v_2,\ldots,v_q$ with respect to $\zeta$, that is, $xv_k=\zeta(x)v_k$, for any $x\in \mf n_\oa$ and $1\leq k\leq q$. This implies that each subspace $\C v_k$ is an one-dimensional $\mf n_\oa$-submodule of $T(M)$  such that $e_{i}\ast v_k = \phi(e_{i})v_k = a_i\zeta(e_{i})v_k = \hat \zeta(e_{i})v_k,$
for any $1\leq k\leq \ell$. Therefore, $T(M)$ is generated by a set of Whittaker vectors with respect to the character $\hat \zeta$. It follows that  $T(M)$ lies in $\mc N^{\mf b}(\hat \zeta)$.

Assume that $\g=\pn$. By the argument above, we may assume, without loss of generality, that the Borel subalgebra $\mf b$ is generated by $e_{ij}$ and $s_{ij}$, for $1\leq i\leq j\leq n$.  Using Lemma \ref{lem::autopn}, the proposition can be proved in the same way as shown before. This completes the proof.

Finally, using Lemma \ref{lem::qnauto}, the case of  $\g=\mf q(n)$ can be established similarly. This completes the proof.
\end{proof}

\vspace{2mm}




\end{document}